\documentclass[a4paper, 11pt]{article}
\usepackage{amssymb, amsmath, mathrsfs, amsthm}%, wasysym}
\usepackage{color}
\usepackage{multicol}
\usepackage{hyperref}
\usepackage[top=3cm, bottom=3cm, left=3cm, right=3cm]{geometry}
\usepackage{float, caption, subcaption}
\usepackage{tikz}
\usepackage{algorithm}
\usepackage{algorithmicx}
\usepackage{algpseudocode}  % ÕâžöºÍalgorithmic²»ŒæÈÝ£¬ÓÃÁËŸÍÒª±šŽí£¬ºÃ¶àÄªÃûÆäÃîµÄŽíÎó£¡£¡£¡£¡£¡

\usepackage{amsfonts}
\usepackage{amscd,color}
\usepackage{amsmath,amsfonts,amssymb,amscd}
\usepackage{indentfirst,graphicx,epsfig}
\input{epsf}
\usepackage{graphicx}
\usepackage{epstopdf}
\usepackage{caption}

\input{epsf}

\newcommand{\bd}{\begin{description}}
\newcommand{\ed}{\end{description}}
\newcommand{\bi}{\begin{itemize}}
\newcommand{\ei}{\end{itemize}}
\newcommand{\be}{\begin{enumerate}}
\newcommand{\ee}{\end{enumerate}}
\newcommand{\beq}{\begin{equation}}
\newcommand{\eeq}{\end{equation}}
\newcommand{\beqs}{\begin{eqnarray*}}
\newcommand{\eeqs}{\end{eqnarray*}}

\definecolor{DarkGreen}{rgb}{0.2, 0.6, 0.3}

\newtheorem{theorem}{Theorem}[section]
\newtheorem{conjecture}{Conjecture}

\newtheorem{lemma}{Lemma}[section]
\newtheorem{definition}{Definition}

\newtheorem{case}{Case}

\newtheorem{claim}{Claim}
\newtheorem{example}{Example}[section]
\newtheorem{remark}{Remark}[section]
\newtheorem{fact}{Fact}

\newtheorem{observation}{Observation}
\begin{document}
\title{Planar Tur\'{a}n number of disjoint union of $C_3$ and $C_4$\footnote{Supported by NSFC No.12201375}}
\author{
 Ping Li\footnote{Corresponding author. School of Mathematics and
Statistics, Shaanxi Normal University, Xi'an, Shaanxi, China. {\tt
lp-math@snnu.edu.cn}}}

\date{}
\maketitle

\begin{abstract}
The {\em planar Tur\'{a}n number} of $H$, denoted by $ex_{\mathcal{P}}(n,H)$, is the maximum number of edges in an $H$-free planar graph.
The planar Tur\'{a}n number of $k\geq 3$ vertex-disjoint union of cycles is a trivial value $3n-6$.
Lan, Shi and Song determine the exact value of $ex_{\mathcal{P}}(n,2C_3)$.
We continue to study planar Tur\'{a}n number of vertex-disjoint union of cycles and obtain the exact value of $ex_{\mathcal{P}}(n,H)$, where $H$ is vertex-disjoint union of $C_3$ and $C_4$.
The  extremal graphs are also characterized.
We also improve the lower bound of $ex_{\mathcal{P}}(n,2C_k)$ when $k$ is sufficiently large.


\end{abstract}

\section{Introduction}
The graphs considered here are simple planar graphs.
A planar graph is a graph which can be
embedded in the plane without crossing edges, and we call such an embedding a {\em plane graph}.
We use $C_k$ and $P_k$ to denote a cycle and a path of order $k$, respectively.
We always use $n$ to denote the order of a graph if there is no confusion.
A face of a plane graph is called an $\ell$-face if the boundary is a $C_k$,
A plane graph in which all faces are $3$-faces is called
a {\em triangulation}.
All terminology and notation not defined in this paper are the same as those
in the textbook \cite{Bondy}.
Given a family of graphs $\mathcal{H}$, a graph is called $\mathcal{H}$-free if it does not contain any member of $\mathcal{H}$ as a subgraph.
The Tur\'{a}n number $ex(n,\mathcal{H})$ is the maximum number of edges in an $n$-vertex $\mathcal{H}$-free graph.
As a classical problem of extremal graph theory, Tur\'{a}n type problem has
attracted much attention.

In 2015, Dowden \cite{Dowden} started to study planar Tur\'{a}n number.
The {\em planar Tur\'{a}n number} of $\mathcal{H}$, denoted by $ex_{\mathcal{P}}(n,\mathcal{H})$, is the maximum number of edges in an $\mathcal{H}$-free planar graph.
In particular, if $\mathcal{H}=\{H\}$, then we write $ex_{\mathcal{P}}(n,\mathcal{H})$ as $ex_{\mathcal{P}}(n,H)$.
Duwden \cite{Dowden} showed that $ex_{\mathcal{P}}(n,C_3)=2n-4$, $ex_{\mathcal{P}}(n,C_4)\leq\frac{15(n-2)}{7}$ for $n\geq 4$ and $ex_{\mathcal{P}}(n,C_5)\leq\frac{12n-33}{5}$ for $n\geq 11$.
Ghosh, Gy\H{o}ri, Martin, Paulos and Xiao \cite{GGMPX} gave that $ex_{\mathcal{P}}(n,C_6)\leq\frac{5n}{2}-7$ for $n\geq18$, and the bound is tight.
Let $\mathcal{C}_k^+$ denote the set of $\Theta$-graphs obtained from $C_k$ by adding a chord.
Lan, Shi and Song \cite{LSS-2} obtained bounds  $ex_{\mathcal{P}}(n,\mathcal{C}_4^+)\leq\frac{12(n-2)}{5}$ for $n\geq 4$, $ex_{\mathcal{P}}(n,\mathcal{C}_5^+)\leq\frac{5(n-2)}{2}$ for $n\geq 5$ and $ex_{\mathcal{P}}(n,\mathcal{C}_6^+)\leq\frac{18(n-2)}{7}$ for $n\geq 7$.
In fact, any graph in $\mathcal{C}_k^+$ consists of two cycles.
Fang, Wang and Zhai \cite{FWZ} gave the tight upper bounds of planar Tur\'{a}n number of intersecting triangles.
Lan, Shi, Song \cite{LSS-5} studied the planar Tur\'{a}n numbers of cubic  graph and some disjoint union of cycles.
For other Planar Tur\'{a}n results, we refer the reader to \cite{CBLS,DWZ,GGPXZ,GGPX,GWZ,LSS-3,LS,LSS-1,LSS-4}.

For two graphs $H_1$ and $H_2$, we use $H_1\cup H_2$ to denote the {\em union} of $H_1$ and $H_2$. If $H_1$ and $H_2$ are vertex-disjoint, then let $H_1\vee H_2$ denote the graph obtained from $H_1$ and $H_2$ by adding all possible edges between $H_1$ and $H_2$.
Since the $n$-vertex planar graph $K_2\vee P_{n-2}$ does not contain $k\geq 3$ pairwise vertex-disjoint  cycles, we have that $ex_{\mathcal{P}}(n,H)=3n-6$ for any $H$ consisting of $k\geq 3$ vertex-disjoint cycles.
Lan, Shi, Song  \cite{LSS-4} determined the exact value of $ex_{\mathcal{P}}(n,2C_3)$.
In this paper, we determine the planar Tur\'{a}n number of vertex-disjoint union of $C_3$ and $C_4$ (for convenience, denote it by $C_3\cup C_4$ throughout this paper) when $n\geq 20$, and  characterize  extremal graphs.
Let $G=\frac{n}{2}K_2$ denote an $n$-vertex graph with $E(G)$ a matching of size $\left\lfloor\frac{n}{2}\right\rfloor$.
The following is our main result.

\begin{theorem}\label{main-1}
If $n\geq 20$, then $ex_\mathcal{P}(n,C_3\cup C_4)=\left\lfloor\frac{5n}{2}\right\rfloor-4$ and the extremal graph is $(\frac{n-2}{2}K_2)\vee K_2$.
\end{theorem}

If $G=(\frac{n-2}{2}K_2)\vee K_2$, then it is clear that $G$ is $C_3\cup C_4$-free.
Therefore, $ex_\mathcal{P}(n,C_3\cup C_4)\geq e(G)= \left\lfloor\frac{5n}{2}\right\rfloor-4$.
Thus, we only need to prove that for any $C_3\cup C_4$-free planar graph $G$, $e(G)\leq \left\lfloor\frac{5n}{2}\right\rfloor-4$ and the equality holds only for $G=(\frac{n-2}{2}K_2)\vee K_2$.

We will prove the theorem in Section \ref{sec-2}.
In Sections \ref{ind} and \ref{sec-4}, we prove some necessary lemmas that will be used in the proof of Theorem \ref{main-1}.
In Section \ref{sec-5}, we give a lower bound of $ex_{\mathcal{P}}(n,2C_k)$ when $k$ is sufficiently large.

\section{Proof of Theorem \ref{main-1}}\label{sec-2}

Let $G$ be an $n$-vertex $C_3\cup C_4$-free plane graph such that $e(G)$ is as large as possible, where $n\geq 20$.
Then
\begin{align}\label{eq-0}
e(G)\geq \left\lfloor\frac{5n}{2}\right\rfloor-4.
\end{align}
Our aim is to show that $e(G)\leq \left\lfloor\frac{5n}{2}\right\rfloor-4$ and the equality holds only for $G=(\frac{n-2}{2}K_2)\vee K_2$.

If an edge $e$ of $G$ lies in the boundaries of exactly two $3$-faces, then we call $e$ an {\em interior edge} of $G$.
We use $E_I(G)$ to denote the set of all interior edges of $G$.
For an edge $e=xy$ of $E_I(G)$, we call $F_1\cup F_2$ a {\em $\Theta$-graph of $xy$} and denote it by $\Theta_{xy}$ or $\Theta_e$, where $F_1$ and $F_2$ are the two $3$-faces of $G$ whose boundaries contain $e$.
Let $f_\ell(G)$ denote the number of $\ell$-faces in $G$ and let $f(G)$ denote the number of faces of $G$.

Let $\{e_1,\ldots,e_t\}\subseteq E_I(G)$.
For a face $F$ of the plane subgraph $H=\bigcup_{i\in[t]}\Theta_{e_i}$, if $F$ is not a $3$-face of some $\Theta_i$, then we call $F$ a {\em pseudo-face} of $H$.
Suppose $F$ is a pseudo-face of $H$ and $u$ is a vertex (resp. $f$ is an edge) of $G$ lying in $F$, but not in the boundary of $F$. Then we call $u$ an interior vertex of $F$ (resp. call $f$ an interior edge of $F$).
If a vertex $u'$ of $G$ is neither an interior vertex (resp. $f'$ is neither an interior edge) of $F$ nor a vertex (resp. nor an edge) of the boundary of $F$, then we call $u'$ an {\em exterior vertex of $F$} (resp. call $f'$ an {\em exterior edge of $F$}).
For convenience, we say that a face $F$ of $G$ contains an edge $e$ if $e$ is in the boundary of $F$.
If the boundary of $F$ is a cycle $u_1u_2\ldots u_ku_1$, then we also use
$u_1u_2\ldots u_ku_1$ to denote the face $F$.
We also use $F$ to denote its boundary if the is no confusion.
\begin{example}
See Figure \ref{example} (the graphs $G'$ is not necessary $C_3\cup C_4$-free), let $e_1=a_1a_2,e_2=b_1b_2,e_3=c_1c_2$ and $e_4=c_2c_3$.
Then $e_i\in E_I(G')$ for each $i\in[4]$ and $H'=\bigcup_{i\in[4]}\Theta_{e_i}$ is the plane subgraph of $G'$.
Observe that $H'$ has two pseudo-faces $F_1=ya_2b_2c_2y$ and $F_2=ya_1b_1c_1c_3y$.
The vertex $x$ is an interior vertex of $F_2$ (resp. an exterior vertex of $F_1$) and the edges $a_1c_1,a_1x,xc_1$ are interior edges of $F_2$ (resp.  exterior edges of $F_1$).
\end{example}
\begin{figure}[h]
    \centering
    \includegraphics[width=350pt]{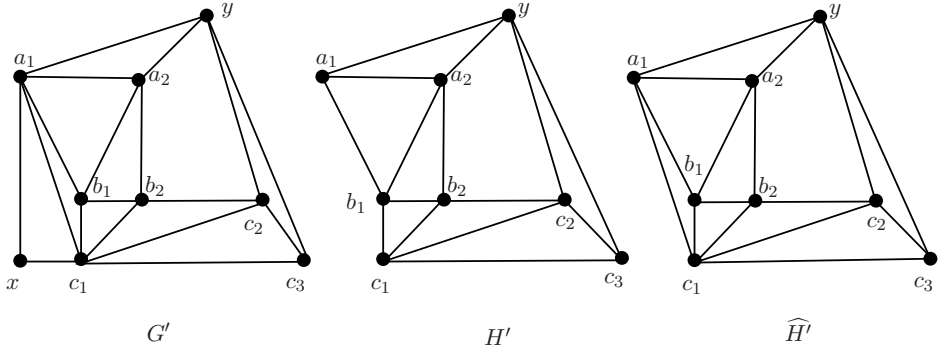}\\
    \caption{The plane graph $G$ and its plane subgraphs.} \label{example}
\end{figure}
It is worth noticing that a pseudo-face of $H$ may also a face of $G$.
For example, see Figure \ref{example}, $F_1$ is a face of $G'$, as well as a pseudo-face of $H'$.
We use $\alpha'(G)$ to denote the size of a maximum matching of $G$.

\begin{lemma}\label{key-lem-1}
If $\alpha'(G[E_I(G)])\leq 1$, then one of the following statements holds.
\begin{enumerate}
  \item $|E_I(G)|\leq 9$.
  \item $f_3(G)+f_4(G)\leq n-1$.
\end{enumerate}
\begin{proof}
If $\alpha'(G[E_I(G)])=0$, then $E_I(G)=\emptyset$ and the first statement holds.
If $\alpha'(G[E_I(G)])=1$, then $G[E_I(G)]$ is either a $K_3$ or a star, and the first statement holds for $G[E_I(G)]=K_3$.
Therefore, suppose $G[E_I(G)]$ is a star with center $u$ and $N_G(u)=\{u_1,u_2,\ldots,u_t\}$.
Without loss of generality, suppose $u_1,u_2,\ldots,u_t$ appear in a clockwise direction.
In addition, let $E_I(G)=\{uu_{c_1},uu_{c_2},\ldots,uu_{c_s}\}$, where $c_1<c_2<\ldots< c_s$.
In order to complete the proof, we only need to show that if $|E_I(G)|=s\geq 10$  then every $\ell$-face contains $u$, where $\ell\in\{3,4\}$.
Suppose to the contrary that there is an $\ell$-face $F$ that does not contain $u$.
If $\ell=3$, then $|V(F)\cap \{u_{c_i}:i\in[s]\}|\leq 3$.
Since $s\geq 10$, it follows that there is an integer $j\in[s]$ such that $|V(F)\cap \{u_{c_{j-1}},u_{c_j},u_{c_{j+1}}\}|=\emptyset$.
Therefore, $F\cup (\Theta_{uu_{c_j}}-uu_{c_j})$ is a $C_3\cup C_4$, a contradiction.
If $\ell=4$, then $|V(F)\cap \{u_{c_i}:i\in[s]\}|\leq 4$.
Since $t\geq 10$, it follows that there is an integer $j\in[s]$ such that $|V(F)\cap \{u_{c_j},u_{c_{j+1}}\}|=\emptyset$.
Therefore, $F\cup (\Theta_{uu_{c_j}}-u_{c_j-1})$ is a $C_3\cup C_4$, a contradiction.
Therefore, every $\ell$-face contains $u$ for $\ell\in\{3,4\}$.
\end{proof}
\end{lemma}

The following lemma is necessary for the proof of Theorem \ref{main-1}, and we will prove it in Section \ref{ind}.
\begin{lemma}\label{key-lem-2}
If $\alpha'(G[E_I(G)])\geq 2$, then $|E_I(G)|\leq \left\lfloor\frac{n}{2}\right\rfloor+4$ and the equality holds only for $G=(\frac{n-2}{2} K_2)\vee K_2$.
\end{lemma}

{\bf Proof of Theorem \ref{main-1}}
Note that $2e(G)=\sum_{i\geq 3}if_i(G)$.
If $G[E_I(G)]$ is a star and $|E_I(G)|\geq 10$, then by Lemma \ref{key-lem-1}, we have that $f_3(G)+f_4(G)\leq n-1$.
Note that
\begin{align*}
2e(G)&\geq 3f_3(G)+4f_4(G)+5(f(G)-f_3(G)-f_4(G))=5f(G)-2f_3(G)-f_4(G).
\end{align*}
By Euler's formula $f(G)=e(G)+2-n$,
$$e(G)\leq \frac{5n+2(f_3(G)+f_4(G))-10}{3}.$$
Since $n\geq 20$ and $f_3(G)+f_4(G)\leq n-1$, it follows that $e(G)< \left\lfloor\frac{5n}{2}\right\rfloor-4$, which contradicts the inequality (\ref{eq-0}).

If $G[E_I(G)]$ is a not star or $G[E_I(G)]$ is a star but $|E_I(G)|\leq 9$,
then by Lemmas  \ref{key-lem-1} and \ref{key-lem-2}, we get that
$|E_I(G)|\leq \left\lfloor\frac{n}{2}\right\rfloor+4$ since $n\geq 20$, and the equality holds only for $G=(\frac{n-2}{2} K_2)\vee K_2$.
Since
$$2e(G)\geq 3f_3(G)+4(f(G)-f_3(G))=3f_3(G)+4(e(G)+2-n-f_3(G)),$$
it follows that $2e(G)\leq f_3(G)+4n-8$.
Let
$$E'=\{e:e\mbox{ is in the boundary of exactly one 3-face of }G\}.$$
Then $E'\cap E_I(G)=\emptyset$ and hence $|E'|\leq e(G)-|E_I(G)|$.
Since $f_3(G)=\frac{|E'|+2|E_I(G)|}{3}$, it follows that
$f_3(G)\leq \frac{e(G)+|E_I(G)|}{3}$.
Therefore,
$$
e(G)\leq\left\lfloor\frac{|E_I(G)|}{5}+\frac{12n}{5}
-\frac{24}{5}\right\rfloor\leq
\left\lfloor\frac{5n}{2}\right\rfloor-4,$$
and the equality holds if and only if $|E_I(G)|= \left\lfloor\frac{n}{2}\right\rfloor+4$.
So, we can further get that $e(G)\leq\left\lfloor\frac{5n}{2}\right\rfloor-4$ and the equality holds only for $G=(\frac{n-2}{2} K_2)\vee K_2$.
The proof thus complete.

The remaining work is to prove Lemma \ref{key-lem-2}.
In Section \ref{ind}, we will introduce the notion of generating graph and prove Lemma \ref{key-lem-2}; in Section \ref{app}, we will prove a key lemma that will be used in the proof of Lemma \ref{key-lem-2}.

\section{Proof of Lemma \ref{key-lem-2}}\label{ind}

The following result is obvious, for otherwise there is a $C_3\cup C_4$, a contradiction.
\begin{lemma}\label{base}
For any two independent edges $e,f$ of $E_I(G)$, $|V(\Theta_e)\cap V(\Theta_f)|\geq 2$.
\end{lemma}

Let $B_0$ be a plane subgraph of $G$, $\{e_1,e_2,\cdots,e_t\}\subseteq E_I(G)$ and $B_i=B_{i-1}\cup \Theta_{e_i}$ for each $0\leq i\leq t$.
If each $e_i$ satisfies that $|V(e_i)\cap V(B_{i-1})|=1$, then we call $B_j$ a {\em generating graph} of $B_0$ for each $0\leq j\leq t$.
Suppose $e=xy,f=ab$ are two independent edges of $E_I(G)$ and $H=\Theta_{xy}\cup \Theta_{ab}$.
Furthermore, assume that $V(\Theta_{xy})=\{x,u_1,y,u_2\}$ and $V(\Theta_{ab})=\{a,v_1,b,v_2\}$.
We now characterize all possible plane subgraphs of $H$ below.
\begin{lemma}\label{lem-base-1}
$H$  is  one of $H_i$, $i\in[0,6]$ $($see Figure \ref{Base=graph}$)$.
\end{lemma}
\begin{proof}
By Lemma \ref{base}, we have that $|V(\Theta_e)\cap V(\Theta_f)|\geq 2$.
Since $e$ and $f$ are independent edges, $f\notin E(\Theta_e)$.
Therefore, $2\leq |V(\Theta_e)\cap V(\Theta_f)|\leq 3$.

Suppose $|V(\Theta_e)\cap V(\Theta_f)|=2$ (say $V(\Theta_e)\cap V(\Theta_f)=\{w_1,w_2\}$).
If $\{w_1w_2\}=E(\Theta_e)\cap E(\Theta_f)$, then $w_1w_2\in E(\Theta_e)-xy$ and $w_1w_2\in E(\Theta_f)-ab$. By symmetry, let $u_2y=av_2$. Then $H=H_0$.
If $w_1w_2\notin E(\Theta_e)$ and $w_1w_2\in E(\Theta_f)$, then $\{w_1,w_2\}=\{u_1,u_2\}$ and $w_1w_2\in \{v_1b,ab\}$ by symmetry.
If $w_1w_2=v_1b$, then $H=H_1$; if $w_1w_2=ab$, then $H=H_3$.
If $w_1w_2\notin E(\Theta_e)$ and $w_1w_2\notin E(\Theta_f)$, then $\{w_1,w_2\}=\{u_1,u_2\}=\{v_1,v_2\}$ and hence $H=H_2$.

Suppose $|V(\Theta_e)\cap V(\Theta_f)|=3$. Then exact one of $x,y$ is in $V(\Theta_{ab})$, say $x\in V(\Theta_{ab})$ and $x=v_2$. By symmetry, either $\{u_1,u_2\}=\{a,b\}$ or $\{u_1,u_2\}=\{b,v_1\}$.
If $\{u_1,u_2\}=\{a,b\}$, then $H=H_5$; if $\{u_1,u_2\}=\{b,v_1\}$, then $H=H_4$.
\end{proof}

\begin{figure}[ht]
    \centering
    \includegraphics[width=420pt]{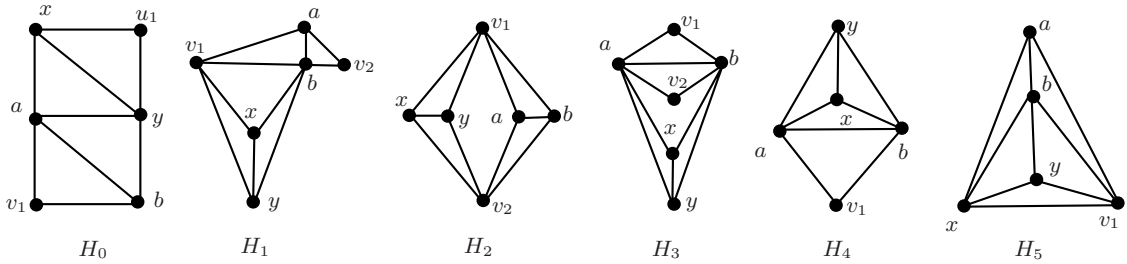}\\
    \caption{All possible plane graphs of $H$.} \label{Base=graph}
\end{figure}

\begin{remark}
If we use $\mathcal{F}_i$ to denote the set of all pseudo-faces of $H_i$, $i\in[0,5]$, then $\mathcal{F}_0=\{axu_1ybv_1a\}$,
$\mathcal{F}_1=\{v_1xbv_1,byv_1av_2b\}$, $\mathcal{F}_2=\{av_2yv_1a,bv_2xv_1b\}$, $\mathcal{F}_3=\{av_2bxa,av_1bya\}$, $\mathcal{F}_4=\{av_1bya\}$ and $\mathcal{F}_5=\{av_1xa,yv_1by\}$.
\end{remark}

Let $\ell^*$ denote the minimum integer of $[0,5]$ such that there are two independent edges of $E_I(G)$, say $e^*$ and $f^*$, such that $\Theta_{e^*}\cup \Theta_{f^*}$  is the plane graph $H_{\ell^*}$ when relabeling vertices (let us just say $\Theta_{e^*}\cup \Theta_{f^*}$ is  an $H_{\ell^*}$ for convenience).
In other word, for each independent edges $e',f'$ of $E_I(G)$, if $\Theta_{e'}\cup \Theta_{f'}$  is an $H_i$, then $i\geq \ell^*$.
We use $H^*$ to denote $\Theta_{e^*}\cup \Theta_{f^*}$.
Without loss of generality, let $e^*=xy$ and $f^*=ab$.

\begin{definition}\label{def-1}
Let $H_{max}$ be a maximal generating graph of $H^*$ and let $\widehat{H}_{max}$ be the plane subgraph induced by $V(H_{max})$.
Let $E_{out}=\{e\in E_I(G):V(e)\cap V(\widehat{H}_{max})=\emptyset\}$.
\end{definition}
It is obvious that if $e\in E(\widehat{H}_{max})-E(H_{max})$, then there is a pseudo-face $F$ of $H_{max}$ such that $e$ is an interior edge of $F$ and $V(e)\subseteq V(F)$.
By the definition of generating graph, we have  that $\partial(\widehat{H}_{max})\cap E_I(G)=\emptyset$, where $\partial(\widehat{H}_{max})$ is the set of edges of $G$ whose one endpoint is in $V(\widehat{H}_{max})$ and the other endpoint is in $V(G)-V(\widehat{H}_{max})$.
Therefore,
\begin{align}
|E_I(G)|=|E_I(G)\cap E(\widehat{H}_{max})|+|E_{out}|.
\end{align}

\begin{lemma}\label{lem-base-2}
If $\ell^*\in \{2,4,5\}$, then any generating graph of $H^*$ is itself.
\end{lemma}
\begin{proof}
The proof proceeds by contradiction.
Suppose there is an edge $e'=zz'$ of $E_I(G)$ such that $z'\in V(H^*)$ and $z\notin V(H^*)$.
Let $V(\Theta_{e'})=\{z,z',c,c'\}$.

\setcounter{case}{0}
\begin{case}
$\ell^*=2$.
\end{case}

By symmetry, let $z'\in \{v_1,x\}$ and $e'$ be an interior edge of the face $xv_2bv_1x$ of $H_2$.

Suppose $z'=v_1$. Then $e'=v_1z$.
Since $e'$ and $xy$, as well as $e'$ and $ab$ are independent edges of $E_I(G)$, it follows from Lemma \ref{base} that $|V(\Theta_{e'})\cap V(\Theta_{xy})|\geq 2$ and $|V(\Theta_{e'})\cap V(\Theta_{ab})|\geq 2$.
Therefore, either $\{c,c'\}=\{x,b\}$ or $v_2\in\{c,c'\}$.
If the former holds, then $\Theta_{e'}\cup \Theta_{ab}$ is an $H_0$, which contradicts the maximality of $\ell^*$.
If the latter holds, without loss of generality, suppose $c'=v_2$.
Since $e'$ is an interior edge of the face $xv_2bv_1x$ of $H_2$, either $c\in\{x,b\}$ or $c$ is an interior vertex of the face $xv_2bv_1x$ of $H_2$.
If $c=x$ (resp. $c=b$), then $zx,zv_1\in E(G)$ (resp. $zb,zv_1\in E(G)$), and hence $abv_2a\cup zxyv_1z$ (resp. $xyv_2x\cup zbav_1z$) is a $C_3\cup C_4$ of $G$, a contradiction.
If $c$ is an interior vertex of the face $xv_2bv_1x$ of $H_2$, then $\Theta_{e'}\cup \Theta_{xy}$  is an $H_1$, which contradicts the minimality of $\ell^*$.

Suppose $z'=x$. Then $e'=xz$.
Since $e'$ and $ab$ are independent edges of $E_I(G)$ and $V(e')\cap V(\Theta_{ab})=\emptyset$, it follows from Lemma \ref{base} that $c,c'\in V(\Theta_{ab})$.
Since $a$ is an exterior vertex of the face $xv_2bv_1x$ of $H_2$, we have that $c,c'\in \{v_1,v_2,b\}$, and hence $\{c,c'\}\cap \{v_1,v_2\}\neq \emptyset$.
Without loss of generality, suppose $c'=v_1$, then $abv_2a\cup yxzv_1y$ is a $C_3\cup C_4$, a contradiction.

\begin{case}
$\ell^*=4$.
\end{case}

The face $av_1bya$  is the unique pseudo-face of $H_4$, and $e'$ is an interior edge of the face $av_1bya$ of $H_4$.
By symmetry, we can assume that $z'\in\{v_1,a,y\}$.

Suppose $z'=v_1$. Then $e'=v_1z$.
Since $e'$ and $xy$ are independent edges of $E_I(G)$ and $v_1,z\notin V(\Theta_{xy})$, it follow from Lemma \ref{base} that $c,c'\in V(\Theta_{x,y})$.
Furthermore, $c,c'\in \{a,b,y\}$.
By symmetry, we can assume that either $\{c,c'\}=\{a,b\}$ or $\{c,c'\}=\{a,y\}$.
If the former holds, then $\Theta_{xy}\cup \Theta_{e'}$  is an $H_2$;
if the latter holds, then $\Theta_{ab}\cup \Theta_{e'}$  is an $H_0$.
Both contradict the minimality of $\ell^*$.

Suppose $z'=a$. Then  $e'=xz$.
Since $a$ is an exterior vertex of the pseudo-face $av_1bya$, $a\notin\{c,c'\}$.
In fact, $b\notin \{c,c'\}$ also.
Otherwise, either $v_1$ or $x$ is an interior vertex of the $3$-face $abza$ of $G$, a contradiction.
Since $e'$ and $xy$ are independent edges of $E_I(G)$, we have that $y\in \{c,c'\}$ (say $y=c'$).
Then $\Theta_{e'}\cup \Theta_{xy}$  is an $H_0$, which contradicts the minimality of $\ell^*$.

Suppose $z'=y$. Then $e'=yz$.
Since $e'$ and $ab$ are independent edges of $E_I(G)$, we can get that $c,c'\in\{a,b,v_1\}$ by Lemma \ref{base}.
If $\{c,c'\}=\{a,b\}$, then $\Theta_{e'}\cup \Theta_{ab}$  is an $H_3$;
if $\{c,c'\}=\{a,v_1\}$ (resp. $\{c,c'\}=\{b,v_1\}$), then $ayza$ (resp. $byzb$) is a $3$-face of $G$ and hence $\Theta_{e'}\cup \Theta_{xa}$ (resp. $\Theta_{e'}\cup \Theta_{xb}$)  is an $H_0$.
Both contradict the minimality of $\ell^*$.

\begin{case}
$\ell^*=5$.
\end{case}

Recall that $bv_1yb,axv_1a$ are all pseudo-faces of $H$.
By symmetry, we can assume that $e'$ is an interior edge of the pseudo-face $axv_1a$ of $H_5$.
Thus, $z'\in\{a,x,v_1\}$.

Suppose $z'=a$. Then $e'=az$.
Since $e'$ and $xy$ are independent edges of $E_I(G)$ and $a,z\notin V(\Theta_{xy})$, we have that $c,c'\in V(\Theta_{xy})$.
Furthermore, we have that $\{c,c'\}=\{x,v_1\}$.
Then $\Theta_{e'}\cup \Theta_{xb}$  is an $H_0$, a contradiction.

Suppose $z'=x$. Then $e'=az$.
Since $e'$ and $ab$ are independent edges of $E_I(G)$, we have that at least one of $a,v_1$ belongs to $\{c,c'\}$.
If $\{c,c'\}=\{a,v_1\}$, then $xv_1\in E_I(G)$ and $\Theta_{xv_1}\cup \Theta_{ab}$  is an $H_3$, a contradiction.
If exact one of $a,v_1$ is in $\{c,c'\}$ (say $c'\in \{a,v_1\}$), then $c$ is an interior vertex of the face $axv_1a$ of $H_5$. Thus, $\Theta_{e'}\cup \Theta_{ab}$  is an $H_0$ if $c'=a$, and $\Theta_{e'}\cup \Theta_{ab}$  is an $H_1$ if $c'=v_1$, a contradiction.

Suppose $z'=v_1$. Then $e'=v_1z$.
Since $e'$ and $xy$ are independent edges of $E_I(G)$, then $x\in \{c,c'\}$ (say $x=c'$).
Thus, $xv_1\in E_I(G)$ and $\Theta_{xv_1}\cup \Theta_{ab}$  is an $H_3$, a contradiction.
\end{proof}

\subsection{generating graphs of $H^*$ when ${\ell}^*=0$}

Suppose $H^*=H_0$  and the labels of vertices are shown in
Figure \ref{Base=graph}.
See Figure \ref{H4}, each $H_{0,i}$, $i\in[5]$, is a generating graph of $H_0$. Moreover,
\begin{enumerate}
  \item $H_{0,1}=H_0\cup \Theta_{yz}$, where $V(\Theta_{yz})=\{c,z,b,y\}$;
  \item $H_{0,2}=H_0\cup \Theta_{yz}$, where $V(\Theta_{yz})=\{c,z,v_1,y\}$;
  \item $H_{0,3}=H_{0,2}\cup \Theta_{yp}$, where $V(\Theta_{yp})=\{p,q,b,y\}$;
  \item $H_{0,4}=H_{0,2}\cup \Theta_{yp}$, where $V(\Theta_{yp})=\{p,q,v_1,y\}$;
  \item $H_{0,5}=H_{0,3}\cup \Theta_{yw}$, where $V(\Theta_{yp})=\{y,u,w,v_1\}$.
\end{enumerate}
Note that $H_{0,1}$ has one pseudo-face $F_A^1=xav_1bzcyu_1x$, $H_{0,2}$ has two pseudo-faces $F_A^2=axu_1yczv_1a$ and $F_B^2=ybv_1y$, $H_{0,3}$ has two pseudo-faces $F_A^3=axu_1yczv_1a$ and $F_B^3=yqpbv_1y$, $H_{0,4}$ has two pseudo-faces $F_A^4=axu_1yczv_1a$ and $F_B^4=ybv_1pqy$, and $H_{0,5}$ has two pseudo-faces $F_A^5=axu_1yczv_1a$ and $F_B^5=yqpbv_1wuy$.
\begin{figure}[ht]
    \centering
    \includegraphics[width=350pt]{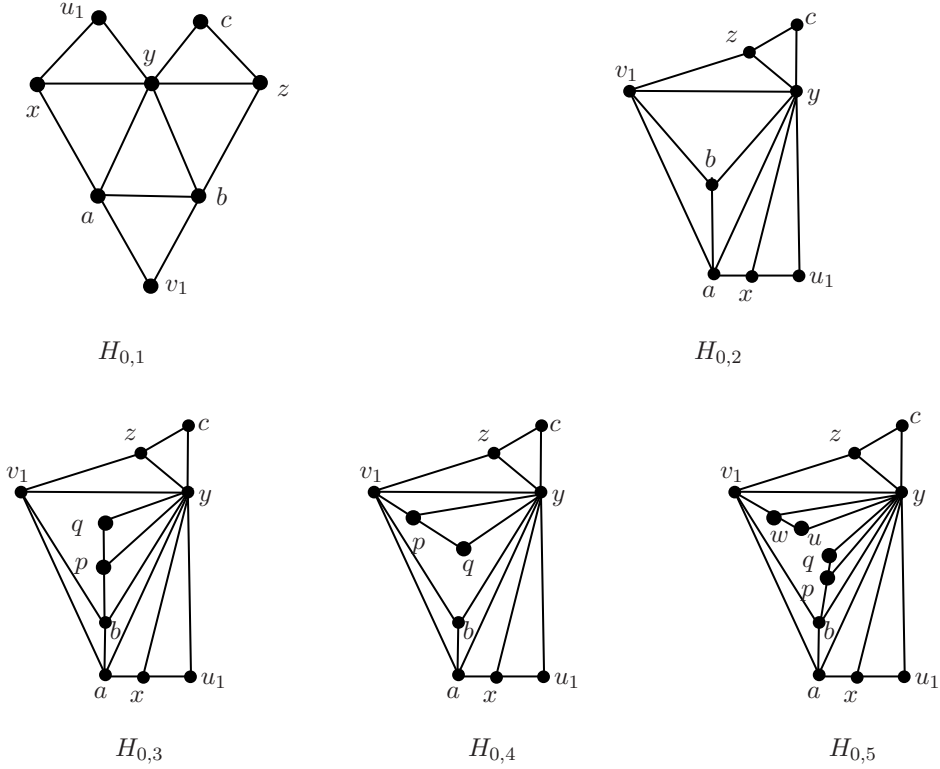}\\
    \caption{generating graphs of $H_0$.} \label{H4}
\end{figure}

We prove that $H_0$ and $H_{0,i}$, $i\in[5]$, are generating graphs of $H_0$ below.

\begin{lemma}\label{lem-H4-1}
Suppose $e=zz'$ is an edge of $E_I(G)$ $($say $V(\Theta_e)=\{z,z',c,c'\})$, $z'\in V(H_0)$ and $z\notin V(H_0)$.
Then $H_0\cup \Theta_{e}$ is a generating graph of $H_0$.
Moreover, $z'\in \{a,y\}$ and $H_0\cup \Theta_{e}$  is either $H_{0,1}$ or $H_{0,2}$.
\end{lemma}
\begin{proof}
By symmetry, we can assume that $z'\in\{u_1,y,b\}$.

Suppose $z'=u_1$. Then $e=u_1z$.
Note that $abv_1a\cup xyzu_1x$ is a $C_3\cup C_4$ of $G$ if $y\in\{c,c'\}$, and $abv_1a\cup xyu_1zx$ is a $C_3\cup C_4$ of $G$ if $x\in\{c,c'\}$.
Therefore, $x,y\notin \{c,c'\}$.
Since $ay$ and $u_1z$ are independent edges of $E_I(G)$ and $u_1,z\notin V(\Theta_{ay})$, it follows from Lemma \ref{base} that $c,c'\in V(\Theta_{ay})$.
Thus $\{c,c'\}=\{a,b\}$, which implies $za,zb\in E(G)$. Then $xyu_1x\cup zbv_1az$ is a $C_3\cup C_4$ of $G$, a contradiction.

Suppose $z'=b$. Then $e=bz$ and $a\notin\{c,c'\}$.
Otherwise, $abza$ is a $3$-face of $G$.
However, one of $v_1$ and $y$ is an interior vertex of the $3$-face $abza$, a contradiction.
Since $xy$ and $bz$ are  independent edges of $E_I(G)$ and $b,z\notin V(\Theta_{xy})$, it follows from Lemma \ref{base} that $c,c'\in V(\Theta_{xy})$.
Since $a\notin\{c,c'\}$, $c,c'\in\{x,y,u_1\}$.
If $\{c,c'\}=\{x,u_1\}$, then $zx,zu_1\in E(G)$ and hence $zxu_1z\cup aybv_1a$ is a $C_3\cup C_4$ of $G$, a contradiction.
If $\{c,c'\}=\{x,y\}$, then $zx,zy\in E(G)$ and hence $abv_1a\cup xu_1yzx$ is a $C_3\cup C_4$ of $G$, a contradiction.
If $\{c,c'\}=\{u_1,y\}$, then $zu_1,zy\in E(G)$ and hence $abv_1a\cup xu_1zyx$ is a $C_3\cup C_4$ of $G$, a contradiction.

Suppose $z'=y$. Then $e=yz$ and $a,x\notin\{c,c'\}$.
Since $ab$ and $yz$ are  independent edges of $E_I(G)$, it follows from Lemma \ref{base} that one of $c,c'$ belongs to $\{v_1,b\}$.
If $\{c,c'\}=\{v_1,b\}$, then $zbv_1z\cup xu_1yax$ is a $C_3\cup C_4$ of $G$, a contradiction.
Note that $u_1\notin \{c,c'\}$.
Otherwise,  $abv_1a\cup xyzu_1x$ is a $C_3\cup C_4$, a contradiction.
Therefore, one vertex of $\{c,c'\}$ belongs to $\{v_1,b\}$ and the other vertex is not in $V(H_0)$.
 Without loss of generality, suppose $c'\in\{v_1,b\}$ and $c\notin V(H_0)$.
If  $c'=b$, then $H_0\cup \Theta_e=H_{0,1}$; if $c'=v_1$, then $H_0\cup \Theta_e=H_{0,2}$.
\end{proof}

\begin{lemma}\label{lem-H4-2}
Let $H\in\{H_{0,1},H_{0,2}\}$.
Suppose $e=pp'$ is an edge of $E_I(G)$ $($say $V(\Theta_e)=\{p,p',q,q'\})$, $p'\in V(H)$ and $p\notin V(H)$.
Then $H\cup \Theta_{e}$ is a generating graph of $H$.
Moreover,
\begin{enumerate}
  \item if $H=H_{0,1}$, then $p'=y$ and $H_1\cup \Theta_{e}$  is $H_{0,3}$ when relabelling vertices;
  \item if $H=H_{0,2}$, then $p'=y$, $e$ is an interior edge of the face $F_B^2$ and $H_1\cup \Theta_{e}$  is  either $H_{0,3}$ or $H_{0,4}$.
\end{enumerate}
\end{lemma}
\begin{proof}
By Lemma \ref{lem-H4-1}, we have that $p'\in \{c,z,y,a\}$.
Suppose $H=H_{0,1}$, then by symmetry, we have that $p'=y$. It is clear that $a,b,x,z\notin\{q,q'\}$.
Since $e$ and $ab$ are independent edges of $E_I(G)$, it follows that $v_1\in \{q,q'\}$ (say $v_1=q'$).
If $u_1=q$ or $c=q$, then $G$ has a $C_3\cup C_4$, a contradiction.
Thus, $c\notin V(H_{0,1})$ and the resulting graph is $H_{0,3}$ when relabelling vertices.

Suppose $H=H_{0,2}$. Recall that $p'\in \{c,z,y,a\}$.
If $p'\in \{z,c\}$, then $e$ is an interior edge of $F_A^2$.
Since $e$ and $xy$, as well as $e$ and $ab$, are independent edges of $E_I(G)$, we have
$\{q,q'\}\in \{a,y\}$.
If $p'=z$, then $y\notin \{q,q'\}$, a contradiction;
if $p'=c$, then $ca\in E(G)$ and $xyu_1x\cup aczv_1a$ is a $C_3\cup C_4$ of $G$, a contradiction.
Thus, assume that $p'\in\{a,y\}$.
If $p'=a$, then $e$ is an interior edge of $F_A^2$. Since $e$ and $yz$, as well as $e$ and $xy$ are independent edges of $E_I(G)$, we have
$y\in\{q,q'\}$, a contradiction.
If $p'=y$ and $e$ is an  interior edge of $F_A^2$, then $v_1,a,b\notin\{q,q'\}$.
However, since $e$ and $ab$ are independent edges of $E_I(G)$, we have that $|V(\Theta_{ab})\cap V(\Theta_e)|\geq 2$ by Lemma \ref{base}, a contradiction.
Now we assume that $e$ is an interior edge of $F_B^2$.
Since $e$ and $ab$ are independent edges of $E_I(G)$, it follows that at least one of
$q,q'$ belongs to $\{v_1,b\}$ (say $q'\in \{v_1,b\}$).
If $\{q,q'\}=\{v_1,b\}$, then $czyz\cup pbav_1p$ is a $C_3\cup C_4$ of $G$, a contradiction.
If $v_1\in \{q,q'\}$ and $b\notin \{q,q'\}$, then  the resulting plane graph is $H_{0,4}$.
If $b\in \{q,q'\}$ and $v_1\notin \{q,q'\}$, then the resulting plane graph is $H_{0,3}$.
\end{proof}
\begin{lemma}\label{lem-H4-3}
If $H\in\{H_{0,3},H_{0,4}\}$, then any generating graph of $H$ is either itself or $H_{0,5}$.
In addition, any generating graph of $H_{0,5}$ is itself.
\end{lemma}
\begin{proof}
Fix $i\in\{3,4\}$ and let $H=H_{0,i}$.
Suppose $e=ww'$ is an edge of $E_I(G)$ (say $V(\Theta_e)=\{w,w',u,u'\}$), $w'\in V(H)$ and $w\notin V(H)$.
We prove $H\cup \Theta_e$ is $H_{0,5}$ below.
Note that $H_{0,2}$ is a plane subgraph of $H_{0,3}$ and $H_{0,4}$.
By Lemma \ref{lem-H4-2}, $w'\in\{y,p,q\}$ and $e$ is an interior edge of $F_B^i$.
If $w'\in\{p,q\}$, then  $e$ and $yx$ are independent edge of $E_I(G)$, and $|V(\Theta_e)\cap V(\Theta_{yx})|\geq 2$ by Lemma \ref{base}.
However, $|V(\Theta_e)\cap V(\Theta_{yx})|\leq 1$, a contradiction.
Thus, $w'=y$ and $e=yw$.
If $i=3$ (resp. $i=4$), then since $e$ and $ab$ are independent edges of $E_I(G)$, we have that $v_1\in\{u,u'\}$ (resp. $b\in\{u,u'\}$). Without loss of generality, say $u'=v_1$ (resp. $u'=b$). Then $w$ is an interior vertex of $F_B^3$ (resp. $F_B^4$), and the resulting graphs is $H_{0,5}$ (resp. the resulting graphs is $H_{0,5}$ when relabelling some vertices).

Now we prove that any generating graph of $H_{0,5}$ is itself.
Otherwise, there is an edge $f=vv'$ of $E_I(G)$ (say $V(\Theta_e)=\{v,v',d,d'\}$), such that $v'\in V(H_{0,5})$ and $v\notin V(H_{0,5})$.
Then $e$ is an interior edge of $F_B^5$ and $v'\in\{p,q,u,w,y\}$.
If $v'\in\{p,q,u,w\}$, then since $e$ and $yx$ are independent edges of $E_I(G)$, we can get a contradiction by Lemma \ref{base}.
If $v'=y$, then since $e$ and $ab$ are independent edges of $E_I(G)$, we have that $|V(\Theta_e)\cap \{v_1,a,b\}|\geq 1$.
However, $v_1,a,b\notin \{d,d'\}$, a contradiction.
\end{proof}

By Lemmas \ref{lem-H4-1}, \ref{lem-H4-2} and \ref{lem-H4-3}, we have the following result.

\begin{observation}\label{obser-H0}
If $H^*=H_0$ and $H$ is a generating graph of $H^*$, then $H\in \{H_0,H_{0,1},\ldots,H_{0,5}\}$.
\end{observation}

\subsection{generating graphs of $H^*$ when ${\ell}^*=1$}

Suppose $H^*=H_1$.
Then $G$ does not contain $H_0$ as a plane subgraph.
Note that $H_1$ has two pseudo-faces, say $A=ybv_2av_1y$ and $B=xbv_1x$.

\begin{figure}[ht]
    \centering
    \includegraphics[width=380pt]{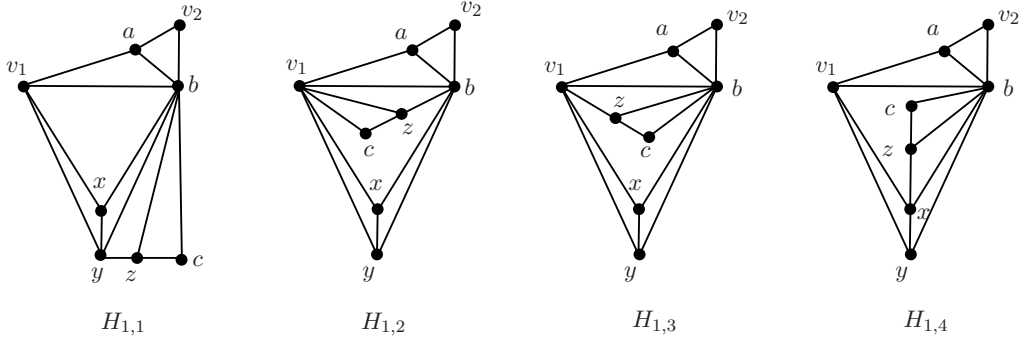}\\
    \caption{generating graphs of $H_1$.} \label{H-1-i}
\end{figure}

\begin{lemma}\label{lem-H11-1}
Suppose $e=zz'$ is an edge of $E_I(G)$ $($say $V(\Theta_e)=\{z,z',c,c'\}$$)$, $z'\in V(H_1)$ and $z\notin V(H_1)$.
Then the generating graph $H_1\cup \Theta_e$   is  $H_{1,3}$ $($see Figure \ref{H-1-i}$)$.
In addition, any generating graph of $H_{1,3}$ is itself.
\end{lemma}
\begin{proof}

We first show that $H_1\cup \Theta_e$  is  $H_{1,3}$.

\setcounter{case}{0}
\begin{case}
$e=zz'$ is an interior edge of $A$.
\end{case}

It is obvious that $c,c'$ are not exterior vertices of $A$.
Since $e=zz'$ is an interior edge of $A$ and $z'\in V(H_1)$, it follows that $z'\in \{v_2,a,v_1,y,b\}$.

Suppose $z'=v_2$. By Lemma \ref{base}, $c,c'\in V(\Theta_{xy})$ since $zz'$ and $xy$ are independent edges of $E_I(G)$.
Since $x$ is an exterior vertex of $A$, it follows that $x\notin \{c,c'\}$.
Therefore, we have that $c,c'\in \{y,v_1,b\}$.
If $v_1\in \{c,c'\}$, then $v_1z\in E(G)$ and hence $xybx\cup v_1av_2zv_1$ is a $C_3\cup C_4$, a contradiction.
Thus, $\{c,c'\}=\{y,b\}$, which implies that $zb\in E(G)$.
Note that $xyv_1x\cup bav_2zb$ is a $C_3\cup C_4$, a contradiction.

Suppose $z'=a$. Then $b,x\notin \{c,c'\}$.
Since $az$ and $xy$ are independent edges of $E_I(G)$, by Lemma \ref{base}, $\{c,c'\}=\{v_1,y\}$ and hence $zv_1,zy\in E(G)$.
Therefore, $abv_2a\cup zv_1xyz$ is a $C_3\cup C_4$, a contradiction.

Suppose $z'=v_1$. Then $b,x\notin \{c,c'\}$. Since $v_1z$ and $xy$ are independent edge of $E_I(G)$, by Lemma \ref{base}, $|V(\Theta_{v_1z})\cap V(\Theta_{xy})|\geq 2$.
Hence $y\in \{c,c'\}$ and $zy\in E(G)$.
So, $abv_2a\cup yxv_1zy$ is a $C_3\cup C_4$ of $G$, a contradiction.

Suppose $z'=y$. Since $yz$ and $ab$ are independent edge of $E_I(G)$, by Lemma \ref{base}, $c,c'\in\{a,b,v_1,v_2\}$.
If $v_1\in\{c,c'\}$, then $abv_2a\cup yxv_1zy$ is  a $C_3\cup C_4$ of  $G$, a contradiction.
If $v_2\in \{c,c'\}$, then $v_2y,v_2z\in E(G)$ and hence $zyv_2z\cup xv_1abx$ is a $C_3\cup C_4$, a contradiction.
Therefore, $\{c,c'\}=\{a,b\}$, which implies that $zb,za\in E(G)$.
Hence $xyv_1x\cup zav_2bz$ is a $C_3\cup C_4$, a contradiction.

Suppose $z'=b$. Then $a,v_1,x\notin \{c,c'\}$.
Since $bz$ and $xy$ are independent edges of $E_I(G)$, by Lemma \ref{base}, $y\in\{c,c'\}$.
Without loss of generality, suppose $c'=y$.
If $c=v_2$, then $xyv_1x\cup bzv_2ab$ is a $C_3\cup C_4$ of $G$, a contradiction.
Therefore,  $c$ is an interior vertex of $A$ and $H_1\cup \Theta_{zb}=H_{1,1}$ (see Figure \ref{H-1-i}).
However, $\Theta_{xy}\cup \Theta_{zb}$ is an $H_0$, a contradiction.

\begin{case}
$e=zz'$ is an interior edge of $B$.
\end{case}

Suppose $z'=x$. Since $xz$ and $ab$ are independent edges of $E_I(G)$, by Lemma \ref{base}, $c,c'\in \{v_1,a,v_2,b\}$.
Note that $a$ and $v_2$ are exterior vertices of $B$.
So, $\{c,c'\}=\{v_1,b\}$ and $abv_2a\cup yxzv_1y$ is a $C_3\cup C_4$ of $G$, a contradiction.

Suppose $z'=v_1$. Since $v_1z$ and $ab$ are independent edges of $E_I(G)$, it follows that $b\in\{c,c'\}$ (without loss of generality, suppose $b=c'$).
If $c=x$, then there is a $C_3\cup C_4$ in $G$, a contradiction.
Thus, $c$ is an interior vertex of $B$.
Therefore, $H_1\cup V(\Theta_e)=H_{1,2}$ (see Figure \ref{H-1-i}).
However, $\Theta_{v_1z}\cup \Theta_{ab}$ is an $H_0$ of $H_{1,2}$, a contradiction.

Suppose $z'=b$. Since $bz$ and $xy$ are independent edges of $E_I(G)$, it follows that either $x\in \{c,c'\}$ or $v_1\in\{c,c'\}$.
Note that $\{x,v_1\}\neq \{c,c'\}$.
Otherwise there is a $C_3\cup C_4$, a contradiction.
If $x\in \{c,c'\}$ (say $x=c'$), then $c$ is an interior vertex of $B$, and hence $H_1\cup \Theta_e=H_{1,4}$ (see Figure \ref{H-1-i}).
However, $\Theta_{xy}\cup \Theta_{bz}$ is an $H_0$, a contradiction.
Thus, $v_1\in \{c,c'\}$ (say $v_1=c'$) and $c$ is an interior vertex of $B$, and hence $H_1\cup \Theta_e=H_{1,3}$.

By above discussion, $H_1\cup \Theta_e$  is an $H_{1,3}$.
Now we prove that any generating graph of $H_{1,3}$ is itself.
Suppose to the contrary that there is an edge of  $E_I(G)$, say $pp'$, such that $p'\in V(H_{1,3})$ and $p\notin V(H_{1,3})$.
Let $V(\Theta_{pp'})=\{p,p',q,q'\}$.
Then $pp'$ is an interior edge of the face $bxv_1zcb$ of $H_{1,3}$ and $p'\in\{z,c,b\}$.
If $p'\in\{z,c\}$, then $pp'$ and $xy$, as well as $pp'$ and $ab$, are independent edges of $E_I(G)$.
Since $p,p'\notin V(\Theta_{xy})$ and $p,p'\notin V(\Theta_{ab})$, it follow from that $\{q,q'\}=V(\Theta_{xy})\cap V(\Theta_{ab})=\{v_1,b\}$.
Since $b\notin \{q,q'\}$ when $p'=z$, we have that $p'=c$ and hence $pv_1\in E(G)$.
Thus, $abv_2a\cup pczv_1p$ is a $C_3\cup C_4$ of $G$, a contradiction.
Therefore, $p'=b$ and hence $v_1,z\notin \{q,q'\}$.
Since $pp'$ and $xy$ are independent edges of $E_I(G)$, it follows that $x\in \{q,q'\}$.
Therefore, $\Theta_{xy}\cup \Theta_{pp'}$ is an $H_0$, a contradiction.
\end{proof}

By Lemma \ref{lem-H11-1}, we have the following result.
\begin{observation}\label{obser-H1}
If $H^*=H_1$ and $H$ is a generating graph of $H^*$, then either $H=H_1$ or $H=H_{1,3}$.
\end{observation}

\subsection{generating graphs of $H^*$ when ${\ell}^*=3$}

Suppose $H^*=H_3$.
Then $G$ does not contain $H_0,H_1$ or $H_2$ as plane subgraphs.
\begin{figure}[ht]
    \centering
    \includegraphics[width=130pt]{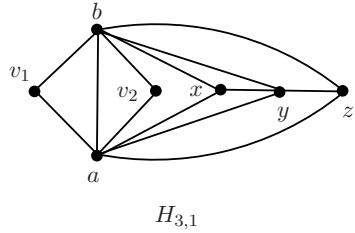}\\
    \caption{An generating graph $H_3$.} \label{H31}
\end{figure}

\begin{lemma}\label{lem-H3}
Each generating graph of $H_3$ is either itself or  $H_{3,1}$ $($see Figure \ref{H31}$)$.
Moreover, any generating graph of $H_{3,1}$ is itself.
\end{lemma}
\begin{proof}
We first prove that each generating graph of $H_3$ is either itself or $H_{3,1}$.
Suppose the generating graph of $H_3$ is not itself. Then there is an edge $e=zz'$  of $E_I(G)$ (say $V(\Theta_e)=\{z,z',c,c'\}$) such that $z'\in V(H_3)$ and $z\notin V(H_3)$.
By symmetry, suppose $e$ is an interior edge of the face $aybv_1a$ of $H_3$ and $z'\in\{y,a,v_1\}$.

Suppose $z'=y$. Then $e=yz$.
Since $yz$ and $ab$ are independent edges of $E_I(G)$ and $y,z\notin V(\Theta_{ab})$, we have that $c,c'\in\{a,b,v_1\}$.
If $\{c,c'\}=\{a,b\}$, then $H_3\cup \Theta_e=H_{3,1}$.
If $\{c,c'\}=\{a,v_1\}$ (resp. $\{c,c'\}=\{b,v_1\}$), then $\Theta_{yz}\cup \Theta_{ab}$  is an $H_1$, a contradiction.

Suppose $z'=v_1$. Then $e=v_1z$. Since $zv_1$ and $xy$ are independent edges of $E_I(G)$ and $v_1,z\notin V(\Theta_{xy})$, we have that  $c,c'\in\{a,b,y\}$.
If $\{c,c'\}\in\{a,b\}$, then $\Theta_{yz}\cup \Theta_{xy}$  is an $H_2$, a contradiction.
If $\{c,c'\}=\{a,y\}$ (resp. $\{c,c'\}=\{b,y\}$), then $\Theta_{v_1z}\cup \Theta_{ab}$  is an $H_0$, a contradiction.

Suppose $z'=a$. Then $e=az$ and $b\notin\{c,c'\}$.
Since $az$ and $xy$ are independent edges of $E_I(G)$, we have that $y\in\{c,c'\}$ (say $y=c'$).
Therefore,  $\Theta_{az}\cup \Theta_{xy}$  is an $H_0$, a contradiction.

Now we prove that any generating graph of $H_{3,1}$ is itself.
Suppose to the contrary that there is an edge $f=pp'$  of $E_I(G)$ (say $V(\Theta_f)=\{p,p',q,q'\}$) such that $p'\in V(H_{3,1})$ and $p\notin V(H_{3,1})$.
By symmetry, suppose $f$ is an interior edge of the face $azbv_1a$ of $H_{3,1}$.
From above discussion, we have that $p'=z$.
Observe that $\{a,b\}=\{q,q'\}$ since $f=zp$ and $xy$ are independent edges of $E_I(G)$.
Therefore, $\Theta_f\cup \Theta_{xy}$  is an $H_2$, a contradiction.
Hence, any generating graph of $H_{3,1}$ is itself.
\end{proof}

\subsection{Counting $|E_I(G)|$}

By Lemmas \ref{lem-base-1}, \ref{lem-base-2} and \ref{lem-H3}, and Observations \ref{obser-H0} and \ref{obser-H1}, we have that $$H_{max}\in\{H_i:i\in[0,5]\}\cup\{H_{0,i}:i\in[5]\}\cup \{H_{1,3},H_{3,1}\}.$$

\begin{lemma}\label{lem-matching}
If $H_{max}\in\{H_i:i\in[0,4]\}\cup\{H_{1,3}\}$, then $E_{out}$ is a matching of $G$;
if $H_{max}\in\{H_{0,i}:i\in[5]\}\cup\{H_5,H_{3,1}\}$, then $E_{out}=\emptyset$.
\end{lemma}
\begin{proof}

Suppose $H_{max}\in(\bigcup_{i\in[0,4]}H_i)\cup\{H_{1,3}\}$.
If $ss'$ is an edge of  $E_{out}$ (say $V(\Theta_{ss'})=\{s,s',r,r'\}$), then by Lemma \ref{base}, we have
\begin{align}\label{eq-app-1}
\{r,r'\}=\left\{
\begin{array}{ll}
V(\Theta_{xy})\cap V(\Theta_{ab}), &  \mbox{ if }H_{max}\in(\bigcup_{i\in[0,3]}H_i)\cup\{H_{1,3}\},\\
\{a,b\}, & \mbox{ if }H_{max}=H_4.
\end{array}
\right.
\end{align}
Furthermore, there is a $3$-face $F$ of $H_{max}$ such that $r'\in V(F)$ and $r\notin V(F)$.
If there are to adjacent edges of $E_{out}$, say $s_1s_2$ and $s_2s_3$, then $F\cup rs_1s_2s_3r$ is a $C_3\cup C_4$, a contradiction.
Therefore, $E_{out}$ is a matching of $G$.

Suppose $H_{max}\in\{H_{0,i}:i\in[5]\}\cup\{H_5,H_{3,1}\}$ and $E_{out}\neq \emptyset$ (say $ss'\in E_{out}$ and $V(\Theta_{ss'})=\{s,s',r,r'\}$).
If  $H_{max}\in\{H_{0,i}:i\in[5]\}$, then by Lemma \ref{base}, we have
$\{r,r'\}\subseteq V(\Theta_{ab})\cap V(\Theta_{xy})\cap V(\Theta_{yz})$.
However, $|V(\Theta_{ab})\cap V(\Theta_{xy})\cap V(\Theta_{yz})|=1$, a contradiction.
For $H_{max}\in\{H_5,H_{3,1}\}$,
\begin{align}\label{eq-app-2}
\{r,r'\}=\left\{
\begin{array}{ll}
\{x,v_1\}, & \mbox{ if }H_{max}=H_5\mbox{ and }ss'\mbox{ is an interior edge of the pseudo-face }axv_1a;\\
\{b,v_1\}, & \mbox{ if }H_{max}=H_5\mbox{ and }ss'\mbox{ is an interior edge of the pseudo-face }byv_1b;\\
\{a,b\}, & \mbox{ if }H_{max}=H_{3,1}.
\end{array}
\right.
\end{align}
Note that there is a $4$-cycle $C$ of $H_{max}$ containing exact one of $r,r'$ (say $r'$).
Thus, $rss'r\cup C$ is a $C_3\cup C_4$ of $G$, a contradiction.
Therefore, $E_{out}=\emptyset$.
\end{proof}

The following conclusion is necessary for the proof of Lemma \ref{key-lem-2}.
For easy of reading, we prove it in the next section.
\begin{lemma}\label{app}
The following statements hold.
\begin{enumerate}
  \item If $H_{max}\in\{H_i:i\in[0,5]\}$, then either $|E_I(G)|< \left\lfloor\frac{n}{2}\right\rfloor+4$ or
      \begin{enumerate}
        \item [1.1.] if $H_{max}\in\{H_0,H_1,H_2\}$, then $|E(\widehat{H}_{\max})\cap E_I(G)|\leq 7$, and the equality holds only if
\begin{align*}
\widehat{H}_{max}=\left\{
\begin{array}{ll}
ay\vee \{xu_1,v_1b\}, & \mbox{ if }H_{max}=H_0;\\
v_1b\vee \{av_2,xy\}, & \mbox{ if }H_{max}=H_1;\\
v_1v_2\vee\{ab,xy\}, & \mbox{ if }H_{max}=H_2.
\end{array}
\right.
\end{align*}
        \item [1.2.] if $H_{max}\in\{H_3,H_4\}$, then $|E_I(G)\cap E(\widehat{H}_{\max})|\leq \left\lfloor\frac{|H_{max}|}{2}\right\rfloor+3$;
        \item [1.3.] if $H_{max}=H_5$, then $|E_I(G)\cap E(\widehat{H}_{\max})|\leq 9$.
      \end{enumerate}
  \item If $H_{max}=H_{1,3}$, then $|E_I(G)\cap E(\widehat{H}_{\max})|\leq \left\lfloor\frac{|H_{max}|}{2}\right\rfloor+3$.
  \item If $H_{max}\in\{H_{0,i}:i\in[5]\}\cup\{H_{3,1}\}$, then $E(\widehat{H}_{\max})\cap E_I(G)|\leq 12$.
\end{enumerate}
\end{lemma}

%\begin{remark}\label{remark-final}
%If $|E_I(G)|=\left\lfloor\frac{n}{2}\right\rfloor+4$, then the plane graph $\widehat{H}_{\max}$ is shown as in Figure \ref{H-G}, where
%\begin{align*}
%\{r,r'\}=\left\{
%\begin{array}{ll}
%\{a,y\}, & \mbox{ if }H_{max}=H_0;\\
%\{v_1,b\}, & \mbox{ if }H_{max}=H_1;\\
%\{v_1,v_2\}, & \mbox{ if }H_{max}=H_2.
%\end{array}
%\right.
%\end{align*}
%\end{remark}

\begin{figure}[ht]
    \centering
    \includegraphics[width=300pt]{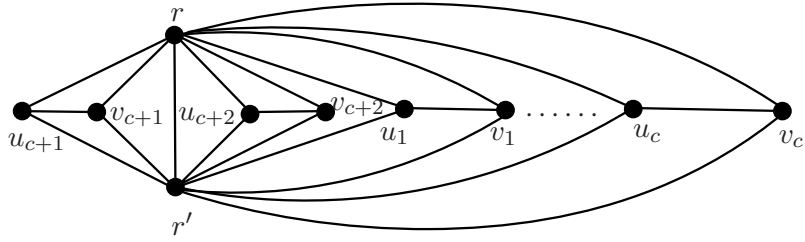}\\
    \caption{A planar embedding of $G^*$.} \label{H-G}
\end{figure}

{\bf Proof of Lemma \ref{key-lem-2}:}
Suppose $H_{max}\in\{H_{0,i}:i\in[5]\}\cup\{H_5,H_{3,1}\}$. By Lemmas \ref{lem-matching} and \ref{app}, we have that $|E_I(G)|=|E_I(G)\cap E(\widehat{H}_{max})|\leq 12< \left\lfloor\frac{n}{2}\right\rfloor+4$ since $n\geq 20$, the result follows.

Suppose $H_{max}=\{H_3,H_4,H_{1,3}\}$.
By Lemmas \ref{lem-matching} and \ref{app}, we have that
$$|E_I(G)|=|E_{out}| +|E_I(G)\cap E(\widehat{H}_{max})|\leq \left\lfloor\frac{n-|H_{max}|}{2}\right\rfloor
+\left\lfloor\frac{|H_{max}|}{2}\right\rfloor+3
<\left\lfloor\frac{n}{2}\right\rfloor+4.$$
The result follows.

Suppose $H_{max}\in\{H_i:i\in[0,2]\}$.
By Lemmas \ref{lem-matching} and \ref{app}, we have that
$$|E_I(G)|=|E_{out}| +|E_I(G)\cap E(\widehat{H}_{max})|\leq \left\lfloor\frac{n-|H_{max}|}{2}\right\rfloor+7
=\left\lfloor\frac{n}{2}\right\rfloor+4.$$
Moreover, if the equality holds, then  $|E_{out}|=\left\lfloor\frac{n-|H_{max}|}{2}\right\rfloor
=\left\lfloor\frac{n-6}{2}\right\rfloor$ and
\begin{align*}
\widehat{H}_{max}=\left\{
\begin{array}{ll}
ay\vee \{xu_1,v_1b\}, & \mbox{ if }H_{max}=H_0;\\
v_1b\vee \{av_2,xy\}, & \mbox{ if }H_{max}=H_1;\\
v_1v_2\vee\{ab,xy\}, & \mbox{ if }H_{max}=H_2.
\end{array}
\right.
\end{align*}
Let $c=\lfloor(n-6)/2\rfloor$.
Since $n\geq 20$, it follows that $c\geq 7$.
Relabel each vertex of $G$ such that
\begin{align*}
\{r,r'\}=\left\{
\begin{array}{ll}
\{a,y\}, & \mbox{ if }H_{max}=H_0;\\
\{v_1,b\}, & \mbox{ if }H_{max}=H_1;\\
\{v_1,v_2\}, & \mbox{ if }H_{max}=H_2,
\end{array}
\right.
\end{align*}
$E_{out}=\{u_1v_1,u_2v_2,\ldots, u_cv_c\}$ and
$\widehat{H}_{max}=rr'\vee\{u_{c+1}v_{c+1},u_{c+2}v_{c+2}\}$.
Let $V(\Theta_{u_iv_i})=\{u_i,v_i,r_i,r'_i\}$ for each $e_i$.
For each $i\in[c]$, since $u_iv_i$ and $u_{c+1}v_{c+1}$, as well as $u_iv_i$ and $u_{c+2}v_{c+2}$ are independent edges of $E_I(G)$, it follows from Lemma \ref{base} that $\{r,r'\}=\{r'_i,r'_i\}$.
Therefore, $G^*=H_{max}\cup (\bigcup_{i\in[c]}\Theta_{u_iv_i})$ is a
$(\left\lfloor\frac{n-2}{2}\right\rfloor K_2)\vee K_2$.
Suppose $U=V(H_{max})\cup(\bigcup_{i\in[c]}V(\Theta_{u_iv_i}))$.
Then $G^*$ is a spanning subgraph of $G[U]$.
If $G^*$ is a proper subgraph of $G[U]$, then $G[U]$ is a graph obtained from $G^*$ by adding some edges $u_iv_j$, where $i\neq j$.
By symmetry, suppose $v_1u_2\in E(G[U])-E(G^*)$.
Then $r'u_3v_3r'\cup ru_1v_1u_2r$ is a $C_3\cup C_4$ of $G[U]$, a contradiction. Therefore, $G[U]=G^*$ is a $(\left\lfloor\frac{n-2}{2}\right\rfloor K_2)\vee K_2$.
If $n$ is even, then $G=G^*=(\frac{n-2}{2}K_2)\vee K_2$.
If $n$ is odd, then there is a vertex $u^*$ of $G$ such that $\{u^*\}=V(G)-V(G^*)$.
Since $e(G)\geq \left\lfloor\frac{n}{2}\right\rfloor-4$, it follows that $d_G(u^*)\geq 2$.
Without loss of generality, suppose $w_1,w_2\in N_G(u^*)$.
By symmetry, we can assume that
$$\{w_1,w_2\}\in \{ \{r,r'\},\{r,u_1\},\{u_1,v_2\},\{u_1,v_1\}\}.$$
If $\{w_1,w_2\}\in \{ \{r,u_1\},\{u_1,v_2\},\{u_1,v_1\}\}$, then it is easy to verify that there is a $C_3\cup C_4$ in $G$, a contradiction.
Therefore, $N_G(u^*)=\{r,r'\}$ and $G$ is a $(\frac{n-2}{2}K_2)\vee K_2$.

\section{Proof of Lemma \ref{app}}\label{sec-4}

We partition $E(\widehat{H}_{max})$ into three parts $E_{diff}$, $Int(H_{max})$ and $Ext(H_{max})$, where $E_{diff}=E(\widehat{H}_{max})-E(H_{max})$,
$Int(H_{max})=E(H_{max})-Ext(H_{max})$ and
$$Ext(H_{max})=\{e:e\mbox{ is an edge of a pseudo-face }F\mbox{ of }H_{max}\mbox{ and }|F|\geq 4\}.$$
It is obvious that each edge of $Int(H_{max})$ is either an interior edge of $H_{max}$ or an edge of a pseudo $3$-face of $H_{max}$.

We divide $E(\widehat{H}_{max})\cap E_I(G)$ into two parts
$$E_A=\{e\in E(\widehat{H}_{max})\cap E_I(G):V(\Theta_e)-V(\widehat{H}_{max})\}\neq\emptyset$$
and
$$E_B=\{e\in E(\widehat{H}_{max})\cap E_I(G):V(\Theta_e)-V(\widehat{H}_{max})\}=\emptyset.$$
For an edge $e\in E_I(G)\cap E(H_{max})$, if $e$ is an interior edge of $H_{max}$, then $e\in Int(H_{max})$; otherwise, exact one $3$-face of $\Theta_e$ is not a $3$-face of $H_{max}$, and we denote it by $F_e$.

\begin{lemma}
Suppose $H_{max}\in\{H_i:i\in[0,5]\}$. Then either $|E_I(G)|<  \left\lfloor\frac{n}{2}\right\rfloor+4$, or the following statements hold.
\begin{enumerate}
  \item If $H_{max}=H_i$, $0\leq i\leq 2$, then $|E_I(G)\cap E(\widehat{H}_{max})|\leq 7$ and the equality holds only if
\begin{align*}
\widehat{H}_{max}=\left\{
\begin{array}{ll}
ay\vee \{xu_1,v_1b\}, & \mbox{ if }H_{max}=H_0;\\
v_1b\vee \{av_2,xy\}, & \mbox{ if }H_{max}=H_1;\\
v_1v_2\vee\{ab,xy\}, & \mbox{ if }H_{max}=H_2.
\end{array}
\right.
\end{align*}
  \item If $H_{max}=H_3$, then $|E_I(G)\cap E(\widehat{H}_{max})|\leq 6$.
  \item If $H_{max}=H_4$, then $|E_I(G)\cap E(\widehat{H}_{max})|=4$.
  \item If $H_{max}=H_5$, then $|E_I(G)\cap E(\widehat{H}_{max})|\leq 9$.
\end{enumerate}
\end{lemma}
\begin{proof}
If $H_{max}=H_5$, then $|E_I(G)\cap E(\widehat{H}_{max})|\leq e(H_5)=9$.
Thus, suppose $H_{max}\in\{H_i:0\leq i\leq 4\}$.
Assume that $|E_I(G)|\geq\left\lfloor\frac{n}{2}\right\rfloor+4$ below.
Since $|H_{max}|\leq6$ and $n\geq 20$, it follows that
$14\leq|E_I(G)|\leq  e(\widehat{H}_{max})+|E_{out}|\leq 12+|E_{out}|$.
Therefore, $|E_{out}|\geq 2$.
Without loss of generality, let $ss',pp'\in E_{out}$ and $V(\Theta_{ss'})=\{s,s',r,r'\}$.
By Lemma \ref{lem-matching}, $ss'$ and $pp'$ are independent edges of $E_I(G)$.

\setcounter{case}{0}
\begin{case}
$H_{max}=H_0$.
\end{case}

Recall equality (\ref{eq-app-1}), we have that $\{r,r'\}=\{a,y\}$.
Thus, $sa,sy\in E(G)$.
Moreover, $sa,sy$ are interior edges of the pseudo-face $byu_1xav_1b$ of $H_0$.
Therefore, the following conclusion is obviously.
\begin{claim}\label{app-clm-1}
Each vertex of $\{x,u_1\}$ and each vertex of $\{v_1,b\}$ are nonadjacent in $G$.
\end{claim}

By Claim \ref{app-clm-1}, $E_{diff}\subseteq \{au_1,bv_1\}$.
\begin{claim}\label{app-clm-2}
Each edge of $E_{diff}$ does not belong to $E_I(G)$.
\end{claim}
\begin{proof}
Suppose to the contrary that one of $au_1,bv_1$ belongs to $E_I(G)$.
By symmetry, suppose $au_1\in E_I(G)$.
Let $V(\Theta_{au_1})=\{a,u_1,d,d'\}$.

Suppose $d,d'\notin V(H_{max})$.
If  $|E_{out}|=2$, then
$$|E_I(G)|\leq e(\widehat{H}_{max})+2\leq e(H_{max})+4=13<\left\lfloor\frac{n}{2}\right\rfloor+4,$$
which contradicts the assumption $|E_I(G)|\geq \left\lfloor\frac{n}{2}\right\rfloor+4$.
If $|E_{out}|>2$, then since $E_{out}$ is a matching of $G$, there is an edge $g\in E_{out}$ such that $d,d'\notin V(e)$.
Note that $y\notin V(\Theta_{au_1})$.
Therefore, $yss'y\cup (\Theta_{au_1}-au_1)$ is a $C_3\cup C_4$ of $G$, a contradiction.

Suppose at most one of $\{d,d'\}$ does not belongs to $V(H_{max})$.
Since $ss',pp'$ are two independent edges of $E_{out}$, we can assume that $d,d'\notin \{s,s'\}$. Then $yss'y\cup (\Theta_{au_1}-au_1)$ is a $C_3\cup C_4$ of $G$, a contradiction.
\end{proof}

It is easy to verify that $xu_1,u_1y,av_1,v_1b\notin E_A$; otherwise there is a $C_3\cup C_4$ in $G$, a contradiction.
By Claim \ref{app-clm-1}, we have that $u_1y\notin E_B$ (resp. $v_1a\notin E_B$); otherwise $u_1v_1$ or $u_1b$ belongs to $E(G)$ (resp. $v_1x$ or $v_1u_1$ belongs to $E(G)$), a contradiction.
Therefore, $u_1y,v_1a\notin E_I(G)$.
Combining  with Claim \ref{app-clm-2},
$E_I(G)\cap E(\widehat{H}_{max})\subseteq E(H_0)-\{u_1y,v_1a\}$.
Furthermore, $E_I(G)\cap E(\widehat{H}_{max})=E(H_0)-\{u_1y,v_1a\}$ indicates that $xu_1,v_1b\in E_I(G)$.
Since $xu_1,v_1b\notin E_A$, it follows that $xu_1,v_1b\in E_B$.
By Claim  \ref{app-clm-1}, $xu_1ax$ and $yv_1by$ are $3$-faces of $G$, and hence $\widehat{H}_{max}=ay\vee \{xu_1,v_1b\}$.
Since $|E(H_0)-\{u_1y,v_1a\}|=7$, we have  that $|E_I(G)\cap E(\widehat{H}_{max})|\leq 7$ and the equality holds only if $\widehat{H}_{max}=ay\vee \{xu_1,v_1b\}$.

\begin{case}
$H_{max}=H_1$.
\end{case}

Note that $E_{diff}\subseteq\{v_1v_2,ay,v_2y\}$.
Recall equality (\ref{eq-app-1}), we have that $\{r,r'\}=\{v_1,b\}$.
If $ay\in E_{diff}$ (resp. $v_2y\in E_{diff}$), then $v_1ss'v_1\cup abxya$ (resp. $v_1ss'v_1\cup v_2bxyv_2$) is a $C_3\cup C_4$ of $G$, a contradiction.
Therefore, $ay,v_2y\notin E(G)$ and $E_{diff}\subseteq\{v_1v_2\}$.
Notice that if $bv_2\in E_B$ then $F_{bv_2}=byv_2b$, and if $v_1y\in E_B$ then $F_{v_1y}$ is either $v_1ayv_1$ or $v_1v_2yv_1$.
Since $ay,v_2y\notin E(G)$, it follows that $bv_2,v_1y\notin E_B$.
In addition,
it is easy to verify that $av_2,bv_2,v_1y\notin E_A$; otherwise there is a $C_3\cup C_4$, a contradiction.
Therefore, $bv_2,v_1y\notin E_I(G)$.

If $v_1v_2\in E(\widehat{H}_{max})\cap E_I(G)$, then $b,x,y\notin V(\Theta_{v_1v_2})$ and hence $bxyb\cup(\Theta_{v_1v_2}-v_1v_2)$ is a $C_3\cup C_4$ of $G$, a contradiction.
Therefore, $v_1v_2\notin E_I(G)$.
Recall that $bv_2,v_1y\notin E_I(G)$.
We have that $E(\widehat{H}_{max})\cap E_I(G)\subseteq E(H_1)-\{bv_2,yv_1\}$.
Moreover, $E(\widehat{H}_{max})\cap E_I(G)=E(H_1)-\{bv_1,yv_1\}$ implies that $av_2\in E_I(G)$.
Since $av_2\notin E_A$, it follows that $av_2\in E_B$ and hence $F_{av_2}=av_1v_2a$.
Therefore, $\widehat{H}_{max}=v_1b\vee \{av_2,xy\}$.
Since $|V(H_1)-\{bv_1,yv_1\}|=7$, we have that $|E(\widehat{H}_{max})\cap E_I(G)|\leq 7$ and the equality holds only if $\widehat{H}_{max}=v_1b\vee \{av_2,xy\}$.

\begin{case}
$H_{max}=H_2$.
\end{case}

Recall equality (\ref{eq-app-1}), we have that $\{r,r'\}=\{v_1,v_2\}$.
Note that $E_{diff}\subseteq\{xb,ya,v_1v_2\}$.
If $xb\in E(\widehat{H}_{max})$ (resp. $ya\in E(\widehat{H}_{max})$), then $ss'v_2s\cup abxv_1a$ (resp. $ss'v_2s\cup xyav_1x$) is a $C_3\cup C_4$, a contradiction.
Therefore, $E_{diff}\subseteq\{v_1v_2\}$.
It is easy to verify that each edge of $E(H_{max})-\{xy,ab\}$ does not belong to $E_A$, for otherwise there is a $C_3\cup C_4$, a contradiction.
By symmetry, we can assume that $v_1v_2$ is an interior edge of the face $av_1yv_2a$ of $H_2$ if $v_1v_2\in E(G)$.
Therefore, whenever $v_1v_2\in E(G)$ or not,  $xv_1,xv_2,bv_1,bv_2\notin E_I(G)$.
Hence
$E(\widehat{H}_{max})\cap E_I(G)\subseteq E(H_{max})\cup\{v_1v_2\}-\{xv_1,xv_2,bv_1,bv_2\}$.
Since $|V(H_{max})\cup\{v_1v_2\}-\{xv_1,xv_2,bv_1,bv_2\}|=7$, we have that $|E(\widehat{H}_{max})\cap E_I(G)|\leq 7$ and the equality holds only if $\widehat{H}_{max}=v_1v_2\vee\{ab,xy\}$.

\begin{case}
$H_{max}=H_3$.
\end{case}

Recall equality (\ref{eq-app-1}), we have that $\{r,r'\}=\{a,b\}$.
Obverse that $E_{diff}\subseteq\{v_2x,v_1y\}$.
If $v_2x\in E(\widehat{H}_{max})$ (resp. $v_1y\in E(\widehat{H}_{max})$),
then $ass'a\cup yxv_2by$ (resp. $ass'a\cup xyv_1bx$) is a $C_3\cup C_4$ of $G$, a contradiction.
Therefore, $E_{diff}=\emptyset$.
Let $D_1=\{av_1,av_2\}, D_2=\{bv_1,bv_2\},D_3=\{bx,by\}$ and $D_4=\{ax,ay\}$.
We prove that $|E(\widehat{H}_{max})\cap E_I(G)|\leq 6$ below.
Suppose to the contrary that $|E(\widehat{H}_{max})\cap E_I(G)|\geq 7$.
Then $|(\bigcup_{i\in[4]}D_i)\cap E_I(G)|\geq 5$.
Thus, either $|(D_1\cup D_3)\cap E_I(G)|\geq 3$ or $|(D_2\cup D_4)\cap E_I(G)|\geq 3$.
By symmetry, suppose $|(D_1\cup D_3)\cap E_I(G)|\geq 3$.
Then we can choose an edge $g_1\in D_1\cap E_I(G)$ and an edge $g_2\in D_3\cap E_I(G)$ such that $g_1,g_2$ are not in the same pseudo-face of $H_{max}$.
Without loss of generality, suppose $g_1=av_1$ and $g_2=bx$.
Since $E_{diff}=\emptyset$, it follows that $E_B=\emptyset$.
Then $av_1,bx\in E_A$ and hence we can find a $C_3\cup C_4$ in $G$, a contradiction.

\begin{case}
$H_{max}=H_4$.
\end{case}

Recall equality (\ref{eq-app-1}), we have that $\{r,r'\}=\{a,b\}$.
Note that $E_{diff}\subseteq\{yv_1\}$.
If $yv_1\in E(\widehat{H}_{max})$, then $ass'a\cup xyv_1bx$ is a $C_3\cup C_4$ of $G$, a contradiction.
Thus, $E_{diff}=\emptyset$ and $H_{max}=\widehat{H}_{max}=H_4$.
Therefore, for each $g\in\{ya,yb,v_1a,v_1b\}$, if $g\in E_I(G)$, then $g\in E_A$.
Next, we prove that for each $g\in \{ya,yb,v_1a,v_1b\}$, $g\notin E_I(G)$.
If $ya\in E_I(G)$, then there is a vertex $w'$ of $V(G)-V(H_{max})$ such that $V(\Theta_{ya})=\{y,a,x,w'\}$.
Since $ss'$ and $pp'$ are independent edges of $E_{out}$, we can assume that $w'\notin \{s,s'\}$.
Thus, $bss'b\cup yxaw'y$ is a $C_3\cup C_4$ of $G$, a contradiction.
So, $ya\notin E_I(G)$.
Similarly, $v_1b\notin E_I(G)$.
If $v_1a\in E_I(G)$ (resp. $v_1b\in E_I(G)$), then $\Theta_{av_1}\cup \Theta_{xb}$ (resp. $\Theta_{ax}\cup \Theta_{v_1b}$)  is  $H_0$.
Since $H_{max}=H_4$ implies $\ell=4$, by the maximality of $\ell$, we get a contradiction.
So, for each $g\in \{ya,yb,v_1a,v_1b\}$, $g\notin E_I(G)$.
Since $\widehat{H}_{max}=H_{max}$, it follows that $E(\widehat{H}_{max})\cap E_I(G)=\{xy,xa,xb,ab\}$ and hence $|E(\widehat{H}_{max})\cap E_I(G)|=4$.
\end{proof}

\begin{lemma}
If $H_{max}=H_{1,3}$, then $|E_I(G)\cap E(\widehat{H}_{max})|\leq7$.
\end{lemma}
\begin{proof}
Let $B_1=\{av_2,bv_2,v_1x,v_1y,cz,cb\}$.
It is easy to verify that each edge of $B_1$ does not belong to $E_A$; otherwise there is a $C_3\cup C_4$, a contradiction.
Note that $E_{diff}\subseteq \{v_1v_2,v_1c,v_2y,ay,xc,xz\}$.
If one  of  $v_2y,ay,xc,xz$ belongs to $E_{diff}$, then there is obviously a $C_3\cup C_4$ in $G$, a contradiction.
Therefore, $v_2y,ay,xc,xz\notin E(G)$ and $E_{diff}\subseteq \{v_1v_2,v_1c\}$.
Furthermore, each edge of $B_2=\{bv_2,bc,v_1x,v_1y\}$ does not belong to $E_B$.
Since $B_2\subseteq B_1$, we have that each edge of $B_2$ is not in $E_I(G)$.

If $av_1v_2a$ (resp. $czv_1c$) is a $3$-face of $G$, then $\Theta_{av_1}\cup \Theta_{bz}$ (resp. $\Theta_{zv_1}\cup \Theta_{ab}$)  is an $H_0$, a contradiction.
Therefore,  we have that $av_2,cz\notin E_B$.
Since $av_2,cz\in B_1$, it follows that $av_2,cz\notin E_I(G)$.
Note that $B_2\notin E_I(G)$ and $B_1=B_2\cup \{av_2,cz\}$.
Then each edge of $B_1$ is not in $E_I(G)$.

We now prove that $v_1v_2,cz_1\notin E_I(G)$.
By symmetry, we only need to prove that $v_1v_2\notin E_I(G)$.
If $v_1v_2\in E_I(G)$, then since $v_2y\notin E(G)$, it follows that $y\notin V(\Theta_{v_1v_2})$.
Since $x,b\notin V(\Theta_{v_1v_2})$, it follows that $|V(\Theta_{v_1v_2})\cap V(\Theta_{xy})|=1$, a contradiction.
So, $v_1v_2\notin E_I(G)$.

By above discussion, we have that $E_I(G)\cap E(\widehat{H}_{max}) \subseteq E(H_{max})-B_1$.
Hence, $|E_I(G)\cap E(\widehat{H}_{max})|\leq 8$, and the equality holds only if $\{v_1a,v_1z,xb,yb\}\subseteq E_I(G)\cap E(\widehat{H}_{max})$.
So, $|E_I(G)\cap E(\widehat{H}_{max})|=8$ implies that $v_1a,bx\in E_I(G)$.
 Since $ay\notin E(G)$, we have that $y\notin V(\Theta_{av_1})$.
 Therefore, $(\Theta_{av_1}-b)\cup(\Theta_{bx}-bx)$ is a $C_3\cup C_4$ of $G$, a contradiction.
Therefore, $|E_I(G)\cap E(\widehat{H}_{max})|\leq 7$.
\end{proof}

\begin{lemma}
If $H_{max}=H_{3,1}$, then $|E_I(G)\cap E(\widehat{H}_{max})|\leq 12$.
\end{lemma}
\begin{proof}
It is obvious that $E_{diff}\subseteq\{xv_2,v_1z\}$. If $\{xv_2,v_1z\}=E_{diff}$, then $axv_2a\cup bv_1zyb$ is a $C_3\cup C_4$ of $G$, a contradiction.
Therefore, at least one of $xv_2,v_1z$ is not an edge of $G$ (say $v_1z\notin E(G)$ by symmetry).
Then $v_1b,v_1a\notin E_B$.
Moreover, $v_1b,v_1a\notin E_I(G)$. Otherwise, either $v_1b$ or $v_1a$ it is an edge of $E_A$, and hence there is a $C_3\cup C_4$ in $G$, a contradiction.
So, $E_I(G)\cap E(\widehat{H}_{max})\subseteq E(H_{max})\cup\{xv_2\}-\{v_1a,v_1b\}$, which implies that
$|E_I(G)\cap E(\widehat{H}_{max})|\leq 12$.
\end{proof}

\begin{lemma}
If $H_{max}=\{H_{0,i}$, $i\in[5]\}$, then $|E_I(G)\cap E(\widehat{H}_{max})|\leq 8$.
\end{lemma}
\begin{proof}
For each $H_{0,i}$, $i\in[5]$, let
$$U=\{uv:uv\mbox{ is a chord of some pseudo-face of }H_{max}\mbox{ and }uv\notin E(H_{max})\}.$$
Then $E_{diff}\subseteq U$.
The following results is easy to verify.
\begin{enumerate}
  \item[(A).] If $H_{max}=H_{0,1}$, then for each $e\in U-\{yv_1\}$, $H_{0,1}\cup e$ contains a $C_3\cup C_4$.
  \item[(B).] If $H_{max}=H_{0,i}$, $2\leq i\leq 5$, then for each $e\in U$, $H_{0,1}\cup e$ contains a $C_3\cup C_4$.
\end{enumerate}
By statements (A) and (B), we have that for each $e\in Ext(H_{0,i})$, $i\in[5]$, if $e\in E_I(G)$, then $e\in E_A$.
On the other hand, the following statements hold.
\begin{enumerate}
  \item[(C).] If $H_{max}\in\{H_{0,1},H_{0,2},H_{0,5}\}$ and there is an edge $e\in Ext(H_{max})\cap E_A$, then $G$ contains a $C_3\cup C_4$.
  \item[(D).] If $H_{max}=H_{0,3}$ and there is an edge $e\in (Ext(H_{max})-\{v_1y\})\cap E_A$, then $G$ contains a $C_3\cup C_4$.
  \item[(E).] If $H_{max}=H_{0,4}$ and there is an edge $e\in (Ext(H_{max})-\{by\})\cap E_A$, then $G$ contains a $C_3\cup C_4$.
\end{enumerate}

By statements (C), (D) and (E), we have that
\begin{itemize}
\item if $H_{max}=H_{0,1}$, $E_I(G)\cap E(\widehat{H}_{max})\subseteq Int(H_{0,1})\cup\{yv_1\}$;
  \item if $H_{max}\in\{H_{0,2},H_{0,5}\}$, then $E_I(G)\cap E(\widehat{H}_{max})\subseteq Int(H_{max})$;
  \item if $H_{max}=H_{0,3}$, $E_I(G)\cap E(\widehat{H}_{max})\subseteq Int(H_{0,3})\cup\{yv_1\}$;
 \item if $H_{max}=H_{0,4}$, $E_I(G)\cap E(\widehat{H}_{max})\subseteq Int(H_{0,4})\cup\{by\}$.
\end{itemize}
Therefore, $|E_I(G)\cap E(\widehat{H}_{max})|\leq 8$ for each $H_{max}\in\{H_{0,i}:i\in[5]$.
\end{proof}

\section{Concluding remark}\label{sec-5}

For $ex_{\mathcal{P}}(n,2C_k)$, Lan, Shi and Song \cite{LSS-5} gave the following lower bound.
\begin{lemma}[\cite{LSS-5}]
Let $n$ and $k$ be positive integers with $n\geq 2k\geq 14$. Let $\varepsilon_1$ and $\varepsilon_2$ be the remainder of
$n-(2k-1)$ when divided by $k-4+\frac{k-1}{2}$ ($k$ is odd) and $k-6+\frac{k}{2}$ ($k$ is even), respectively.
\begin{enumerate}
  \item If $n$ is odd, then $ex_{\mathcal{P}}(n,2C_k)=3n-6$ for all $n\leq 3k-4$, and
  $$ex_{\mathcal{P}}(n,2C_k)\geq \left(3-\frac{1}{k-4+\lfloor\frac{k}{2}\rfloor}\right)n+
\frac{5+\varepsilon_1}{k-4+\lfloor\frac{k}{2}\rfloor}
-\frac{17}{3}+\max\{1-\varepsilon_1,0\}$$ for all $n\geq 3k-3$.
  
\item If $n$ is even, then $ex_{\mathcal{P}}(n,2C_k)=3n-6$ for all $n\leq 3k-7$, and
  $$ex_{\mathcal{P}}(n,2C_k)\geq \left(3-\frac{1}{k-6+\lfloor\frac{k}{2}\rfloor}\right)n+
\frac{7+\varepsilon_2}{k-6+\lfloor\frac{k}{2}\rfloor}
-\frac{17}{3}+\max\{1-\varepsilon_2,0\}$$ for all $n\geq 3k-6$.
\end{enumerate}
\end{lemma}

\begin{figure}[ht]
    \centering
    \includegraphics[width=350pt]{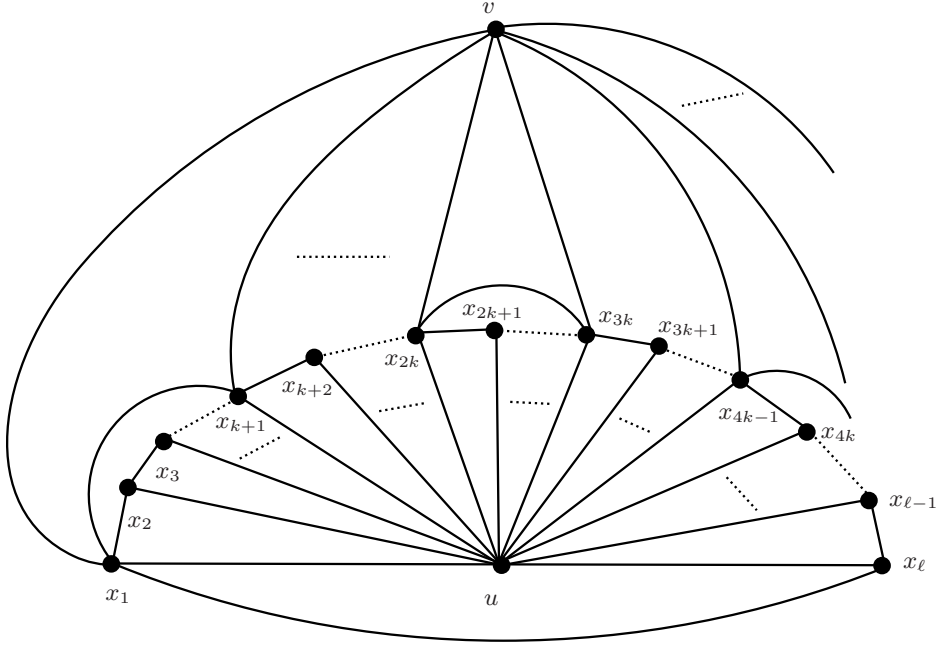}\\
    \caption{The planar graph $G_{k,\ell}$.} \label{2ck}
\end{figure}

We will improve the lower bound of $ex_{\mathcal{P}}(n,2C_k)$ when $n$ and $k$ are sufficiently large.
We first construct a family $\mathcal{G}_{k,\ell}$ of $2C_k$-free planar graphs below, where $\ell \geq 2k$.
Let $C=x_1x_2\ldots x_\ell x_1$ be a cycle and let $W=u\vee C$.
For $1\leq i\leq \ell/(2k-1)$, let $T_i$ be a triangle with $V(T_i)=\{v,x_{m_i},x_{n_i}\}$, where $m_i=(2k-1)(i-1)+1$ and $n_i=m_i+k$.
Let $G_{k,\ell}$ be a plane graph of $(\bigcup_{1\leq i\leq \ell/(2k-1)}T_i)\cup W$ (see Figure \ref{2ck}) such that each $T_i$ and $S_j=x_jx_{j+1}ux_j$ are $3$-faces, where $1\leq i\leq \ell/(2k-1)$ and $j\in[\ell]$.
It is clear that all but at most one face of $G_{k,\ell}$ is either a $3$-face or a $(k+1)$-face.
Suppose $T'$ and $T''$ are triangulations such that
\begin{enumerate}
  \item the maximum length of cycles in $T'$ is at most $k-1$;
  \item the maximum length of cycles in $T''$ is at most $k$.
  Moreover, if $T''$ contains $k$-cycles, then there exists a vertex $u^*$ of $T''$ such that any $k$-cycle of $T''$ contains $u^*$.
\end{enumerate}
Let $\mathcal{G}_{k,\ell}$ be the set of plane graphs obtained from $G_{k,\ell}$ by replacing each $3$-face $T_i$ with $T'$ and replacing each $3$-face $S_i$ with $T''$ (if the special vertex $u^*$ exists, then let $u=u^*$).
For example, let $k=4$, $T'$ be a triangle and $T''$ be a planar embedding of $K_4$.
The resulting graph $G_0\in \mathcal{G}_{4,\ell}$ is shown in Figure \ref{shi-counter-example-1}.

\begin{figure}[ht]
    \centering
    \includegraphics[width=220pt]{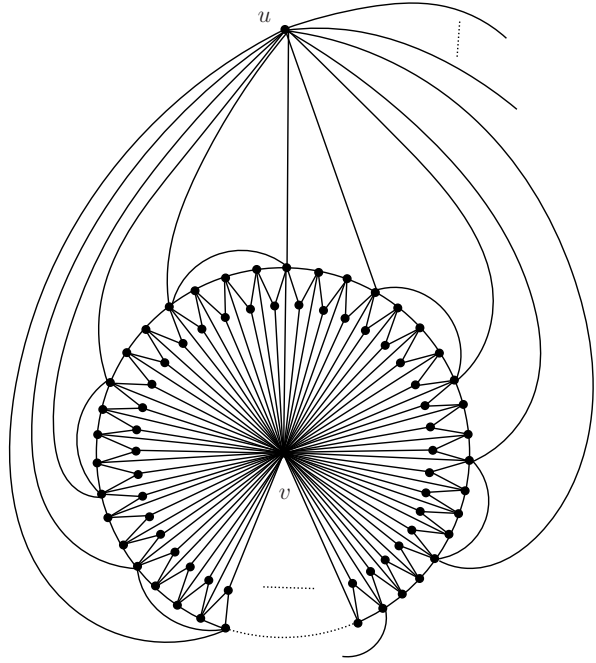}\\
    \caption{The planar graph $G_0$.} \label{shi-counter-example-1}
\end{figure}

\begin{lemma}\label{2k-free}
Any graph of $\mathcal{G}_{k,\ell}$ is $2C_k$-free.
\end{lemma}
\begin{proof}
Let $G\in \mathcal{G}_{k,\ell}$ be a graph obtained from $G_{k,\ell}$ by replacing each $3$-face $T_i$ with $T'$ and replacing each $3$-face $S_i$ with $T''$, where $T'$ and $T''$ are triangulations defined above.
For each $T_i$ and $S_i$ of $G_{k,\ell}$, we denote the replaced triangulations by $T'_i$ and $T''_i$, respectively.
Let $\mathcal{T}=\{T''_i: i\in[\ell]\}\cup \{T'_i:1\leq i\leq \ell/(2k-1)\}$.
Let $I(T'_i)=V(T'_i)-\{v,x_{n_i},x_{m_i}\}$ for each $1\leq i\leq \ell/(2k-1)$ and let $I(T''_j)=V(T''_j)-\{x_j,x_{j+1},u\}$ for each  $j\in[\ell]$.
For each $T\in \mathcal{T}$ and any two vertices $y,z\notin I(T)$, let $d_G(T;y,z)$ denote the length of a shortest path in $G-I(T)$ joining $y$ and $z$.
The following result is obvious.
\begin{fact}\label{fact}
If $y,z\in V(T'_i)-I(T'_i)$ for $1\leq i\leq \ell/(2k-1)$, then $d_{G-u}(T;y,z)\geq k$;
if $\{y,z\}=\{x_i,x_{i+1}\}$, then $d_{W-u}(T;y,z)=\ell-1$.
\end{fact}

In order to prove that $G$ is $2C_k$-free, we only need to prove that any $k$-cycle of $G$ contains $u$.
Suppose to the contrary that there is a $k$-cycle $D$ of $G$ does not contain $u$.
Then $D$ is not a subgraph of any $T'_i$ and $T''_j$.
Therefore, there is a $T\in \mathcal{T}$ such that
\begin{enumerate}
  \item there are two vertices $y,z$ of $V(D)\cap(V(T)-I(T))$;
  \item there is a path $P$ in $D$ such that $|P|\geq 3$, $y,z$ are endpoints of $P$ and each inner vertex of $P$ is not in $V(T)$.
\end{enumerate}
If $T=T'_i$ for some $1\leq i\leq \ell/(2k-1)$, then  $y,z\subseteq \{v,x_{n_i},x_{m_i}\}$.
Since $d_{G-u}(T'_i;y,z)\geq k$ by Face \ref{fact}, $D$ contains at least $k+1$ vertices, a contradiction.
If $T$ can not be any $T'_i$, then $D$ is contained in $W$ and $T=T''_j$ for some $j\in [\ell]$.
Since $u\notin V(D)$, it follows that $\{y,z\}=\{x_j,x_{j+1}\}$.
By Face \ref{fact}, $d_{W-u}(T;y,z)=\ell-1$.
Recall that $\ell\geq 2k$.
Therefore, $|D|>k$, a contradiction.
\end{proof}

Suppose $|T'|=n_1$, $|T''|=n_2$ and $k\geq 4$.
Let $G\in\mathcal{G}_{k,\ell}$.
Then
\begin{align*}
|G|&=n=\left\lfloor\frac{\ell}{2k-1}\right\rfloor(n_1-3)+\ell(n_2-2)+2
\end{align*}
and
\begin{align*}
e(G)&=\left\lfloor\frac{\ell}{2k-1}\right\rfloor(3n_1-6)+\ell(3n_2-7).
\end{align*}
Therefore,
$$e(G)=3n-\ell+3\left\lfloor\frac{\ell}{2k-1}\right\rfloor-6.$$
Let $T$ be a maximum $C_k$-free triangulation and let $n_1=n_2=|T|$.
Then
$$\ell =\frac{2k-1}{|T'|-3+(2k-1)(|T|-2)}n+c_k,$$
and hence
\begin{align}\label{ineq-log}
ex_{\mathcal{P}}(n,2C_k)\geq e(G)=\left(3-\frac{1}{\frac{k|T|}{k-2}-\frac{4k+1}{2k-4}}\right)n+d_k,
\end{align}
where $-1\leq c_k\leq 1$ and $-10\leq d_k\leq -5$.

\begin{lemma}[\cite{MM}]\label{lem-log32}
For each positive integer $k$ there exists a $3$-connected plane triangulation $G_k$ with $|G_k|=\frac{3^{k+1}+5}{2}$ and with longest cycle of length less than $\frac{7}{2}|G_k|^{\log_32}$.
\end{lemma}

By Lemma \ref{lem-log32}, we can assume that $T'=T''=G_t$, where
$$t=\left\lfloor\log_3\left(2\cdot(2k/7)^{\log_23}-5\right)\right\rfloor-1.$$
Set $\alpha=3^{-\varepsilon}\cdot(\frac{2}{7})^{\log_23}$, where
$$\varepsilon=\log_3\left(2\cdot(2k/7)^{\log_23}-5\right)-
\left\lfloor\log_3\left(2\cdot(2k/7)^{\log_23}-5\right)\right\rfloor.$$
Then $|G_t|=\alpha k^{\log_23}+\frac{5(1-3^{-\varepsilon})}{2}$ and the length of a longest cycle of $G_t$ less than $k$.
Recall inequality (\ref{ineq-log}), we have the following result.

\begin{theorem}
If $n\gg k$ and $k\geq 4$, then
\begin{align*}
ex_{\mathcal{P}}(n,2C_k)\geq \left(3-\frac{1}{\frac{\alpha}{k-2}k^{1+\log_23} +\beta}\right)n-\gamma_k,
\end{align*}
 where $-5<\beta<2$ and $\gamma_k$ is a real number depending only on $k$.
\end{theorem}

%\begin{lemma}
%For an integer $k\geq 4$, if $n$ is sufficient large, then
%$$ex_{\mathcal{P}}(n,2C_k)\geq \left(3-\frac{1}{k-\frac{1}{k-2}}\right)n
%+\frac{\varepsilon(k-4)-2}{k-\frac{1}{k-2}}-9,$$
%where
%$\varepsilon=\frac{\ell}{2k-1}-
%\left\lfloor\frac{\ell}{2k-1}\right\rfloor$.
%\end{lemma}
%\begin{proof}
%Let $T'$ be a triangulation on $k-1$ vertices and $T''$ be a triangulation on $k$ vertices.
%Let $G$ a plane graph obtained from $G_{k,\ell}$ by replacing each $3$-face $T_i$ with $T'$ and replacing each $3$-face $S_i$ with $T''$.
%By Lemma \ref{2k-free}, $G$ is $2C_k$-free.
%Notice that
%$$n=|G|=\left(\frac{k-4}{2k-1}+k-2\right)\ell+2-\varepsilon (k-4)$$
%and
%$$e(G)=\left(\frac{3k-9}{2k-1}+3k-7\right)\ell-\varepsilon (3k-9).$$
%Then $3n=e(G)+(1-\frac{3}{2k-1})\ell+6+3\varepsilon$.
%Note that
%$$\ell=\frac{(2k-1)(n-2+\varepsilon(k-4))}{2k^2-4k-2}.$$
%Therefore,
%$$e(G)=\left(3-\frac{1}{k-\frac{1}{k-2}}\right)n
%+\frac{\varepsilon(k-4)-2}{k-\frac{1}{k-2}}-9.$$
%\end{proof}

For $k=4$, we propose the  following  conjecture.
\begin{conjecture}\label{conj}
If $n\geq 23$, then $ex_{\mathcal{P}}(n,2C_4)\leq \frac{19}{7}(n-2)$ and the bound is tight for $14|(n-2)$.
\end{conjecture}

See Figure \ref{shi-counter-example-1}, $G_0$ is a $2C_4$-free planar graph with
$e(G)=\frac{19}{7}(n-2)$ when $14|(n-2)$.

\end{document}